\documentclass[12pt]{amsart}      
\usepackage{amssymb,enumerate}
\usepackage{amsmath}
\usepackage{setspace}
\usepackage{tcolorbox}

\usepackage{float,enumitem}
\usepackage{mathscinet}

\usepackage{amsmath,amssymb,amsfonts,amsthm}

\newtheorem{theorem}{Theorem}[section] 
\newtheorem{lemma}[theorem]{Lemma} 
\newtheorem{prop}[theorem]{Proposition}

\theoremstyle{definition} 
\newtheorem{remark}[theorem]{Remark} 
\newtheorem{definition}[theorem]{Definition} 
\newtheorem{example}[theorem]{Example}

\newcommand{\U}{\mathcal{U}}

\newcommand{\card}{\operatorname{card}}
\newcommand{\pH}{p^{\mathrm{H}}}
\newcommand{\dH}{d^{\mathrm{H}}}

\newcommand{\dB}{d^{\mathrm{B}}}

\newcommand\N{\mathbb{N}}
\newcommand\FF{\mathcal{F}}

\newcommand{\id}{\mathrm{id}}
\newcommand{\be}{\begin{equation}}
\newcommand{\ee}{\end{equation}}

\newcommand{\tv}{\tau_{\mathrm{V}}}

\newcommand{\KK}{\mathcal K}

\newcommand{\vp}{\varphi}
\newcommand{\B}{{\mathcal{B}}}

\newcommand{\A}{{\mathcal{A}}}

\newcommand{\ep}{{\varepsilon}}

\newcommand\restr[1]{\raisebox{-.1ex}{$\big\vert$}_{#1}}

\newcommand{\orbn}{{\operatorname{Orb}_n}}
\newcommand\orb[1]{{\operatorname{Orb}_{#1}}}
\newcommand{\KT}{\operatorname{KTRT}}

\newcommand{\ANR}{\operatorname{ANR}}
\renewcommand{\H}{\operatorname{H}}

\newcommand{\CM}{\operatorname{CMM}}

\newcommand{\hKT}{h_{\operatorname{KTRT}}}

\newcommand{\hCM}{h_{\operatorname{CMM}}}
\newcommand{\hrsp}{h_{\rho}^{\operatorname{span}}}
\newcommand{\hrse}{h_{\rho}^{\operatorname{sep}}}

\newcommand{\hi}{h_{\mathfrak{i}}}
\newcommand{\hU}{h_{U}}
\newcommand{\hL}{h_{L}}
\newcommand{\hH}{h_{\operatorname{H}}}
\newcommand{\hB}{h_{\operatorname{B}}}

\newcommand{\hHsp}{h_{\H}^{\operatorname{span}}}
\newcommand{\hHse}{h_{\H}^{\operatorname{sep}}}

\newcommand{\pnH}{p_n^{\mathrm{H}}}

\newcommand{\spa}{\operatorname{span}}
\newcommand{\sepp}{\operatorname{sep}}

\newenvironment{rrcases}{\left.\begin{aligned}}{\end{aligned}\right\rbrace}

\usepackage{mathscinet}

\begin{document}

\begin{abstract}
We will consider various definitions of topological entropy for multivalued nonautonomous dynamical systems in compact Hausdorff spaces. Some of them can deal with arbitrary multivalued maps, i.e. when no restrictions are imposed on them. For upper semicontinuous multivalued maps, we still investigate especially their relationship to the parametric topological entropy of the induced (possibly discontinuous) dynamical systems in hyperspaces. A comparison of various sorts of definitions can allow us, rather than direct calculations, to make easier the entropy estimates. Several illustrative examples are supplied.
\end{abstract}


\thanks{The authors were supported by the Grant IGA\_PrF\_2024\_006 ``Mathematical Models'' of the Internal Grant Agency of Palack\'y University in Olomouc.}


\title[Parametric topological entropy of possibly discontinuous maps]{Parametric topological entropy of possibly discontinuous maps in compact Hausdorff spaces and hyperspaces}

\author[J. Andres]{Jan Andres}
\address{Department of Mathematical Analysis and Applications of Mathematics, Faculty of Science, Palack\'y University, 17. listopadu 12, 771 46 Olomouc, Czech Republic}
\author[P. Ludv\'\i k]{Pavel Ludv\'\i k}
\email{pavel.ludvik@upol.cz}

\maketitle

\section{Introduction}
The classical topological entropy is an important topological invariant which measures the complexity of dynamical systems behaviour. The single-valued maps under consideration in its standard definitions are at least continuous (cf. \cite{AKM, Bo,Di,Ho,KS}). The extensions to multivalued maps concern mainly with upper semicontinuous maps (cf. \cite{AK,An,AL1,AL2,CP,EK,KT,RT,WZZ}) or less frequently with lower semicontinuous maps (cf. \cite{An,AL1,AL2}). Multivalued dynamics play with this respect an increasing role, especially in economic applications (see e.g. \cite{RS,St,Vo1,Vo2}, and the references therein).

On the other hand, as observed e.g. in \cite{JS,KLP,Ma,MZ,NS,Pw,PLB}, a less amount of regularity (like a piece-wise continuity or an almost continuity) can be required in the related definitions of topological entropy of single-valued maps, or they can be quite arbitrary (cf. \cite{Ci,SS}). Since even multivalued maps can be arbitrary (cf. \cite{AL4,CMM,SS}), the following natural questions arise:
\begin{enumerate}
	\item[(i)] Are the definitions of topological entropy $h(\vp)$ for arbitrary (single-valued or multivalued) maps $\vp:X\multimap X$ correct?
	\item[(ii)] If so, can the possibly discontinuous hypermaps $\vp^*:\KK(X)\to\KK(X)$ in the hyperspace $\KK(X):=\{K\subset X: \mbox{ $K$ is a nonempty compact subset of $X$}\}$, which are induced by upper semicontinuous maps $\vp:X\to \KK(X)$, satisfy the inequality $h(\vp)\leq h(\vp^*)$?
\end{enumerate}

Let us note that such an inequality was already obtained for single-valued continuous maps in e.g. \cite{BS,BV,CL,QWW,KO,LWW,MHL,Sh2}, and for multivalued maps in \cite{An,AL1,AL2}. As far as we know, there are no such results at the absence of continuity.

Since for the calculations and estimates of values of topological entropy of (possibly arbitrary) multivalued maps and their single-valued selections the mutual relationship among various sorts of definitions can make the process easier and more transparent, we will also pay attention to the comparison of definitions.

Hence, our main aim will be to answer at least partly the problems indicated above, provided $X$ is a compact Hausdorff space and the multivalued dynamical systems are, more generally than above, nonautonomous. It means that, instead of a single multivalued map $\vp:X\multimap X$, we would like to consider a sequence $\vp_{0,\infty}$ of multivalued maps $\vp_j:X\multimap X$, $j\in\N\cup\{0\}$, whose compositions $\vp_{n-1}\circ \ldots \circ \vp_0$ replace the role of the $n$-th iterates of $\vp$. Such an entropy will be called \emph{parametric} (whence the title).

Our paper is organized as follows. After some technical preliminaries, various definitions of topological entropy will be considered for arbitrary multivalued nonautonomous maps and, separately, for semicontinuous maps. Then possibly discontinuous hypermaps induced by nonautonomous upper semicontinuous maps will be investigated with respect to the inequality in (ii). For certain subclasses of multivalued upper semicontinuous maps, some further definitions of parametric topological entropy will be still examined for the same goal like those in \cite{CMM}. Some simple illustrative examples and comments will be supplied.

\section{Preliminaries}

The topological spaces under our consideration will be at least compact Hausdorff.

A collection $\A$ of subsets of a topological space space $(X,\tau)$ is a \emph{cover} of $X$ if their union is all of $X$, i.e. $X = \bigcup\A$. A collection $\A$ is an \emph{open cover} of $X$ if $X = \bigcup\A$ and additionally $\A\subset\tau$. If $\B\subset\A$ is also a cover of $X$, then we say that the cover $\A$ has a \emph{subcover} $\B$ of $X$.

A topological space $(X,\tau)$ is \emph{compact} if every open cover of $X$ has a finite subcover. Every compact Hausdorff space is well known to be uniformizable.

More concretely, every compact Hausdorff space is a complete uniform space with respect to the unique uniformity compatible with the topology. Thus, all the properties of uniform spaces can be employed in compact Hausdorff spaces.

Let us therefore recall the basic properties of uniform spaces with their relation to compact Hausdorff spaces.

\begin{definition}
If $\mathcal{P}$ is a family of pseudometrics on a set $S$ and $p$ is a pseudometric on $S$, then we write $p\ll\mathcal{P}$, provided
\begin{equation*}
\forall\ep>0\, \exists p_\ep\in\mathcal{P} \,\exists\delta>0 \,\forall x,y\in S: p_\ep(x,y)<0 \Rightarrow p(x,y)<\ep.
\end{equation*}

\end{definition}

\begin{definition}\label{d:uniform1}
A \emph{uniform structure} (or a \emph{uniformity}) on a nonempty set $S$ is a family $\U$ of pseudometrics on $S$ with the properties:
\begin{enumerate}
	\item[(U1)] If $p_0,p_1\in\U$, then $\max\{p_0,p_1\}\in\U$.
	\item[(U2)] If $p$ is a pseudometric on $S$ such that $p\ll\U$, then $p\in\U$.
	\item[(U3)] If $x,y\in S$, $x\neq y$, then there is $p\in\U$ such that $p(x,y)>0$.
\end{enumerate}
A \emph{uniform space $X$} is a nonempty set $S$ (set of points of $X$) together with a uniformity on $S$.
\end{definition}

Let $\U = \U(X)$ denote the uniformity of $X$ (i.e., the family of pseudometrics). Denoting, for $p\in\U$, $x\in X$ and $r>0$,
\begin{equation*}
B_p(x,r):= \{y\in X: p(x,y)<r\},
\end{equation*}
the collection of sets $B_p(x,r)$, where $p\in\U$ and $x\in X$, $r>0$, is a base of a Hausdorff topology on the set of points of $X$.

The following generalization of the Lebesgue covering lemma will be useful in the sequel.

\begin{lemma}[see {\cite[Theorem 33, and the paragraph before]{Ke}}]\label{l:Lebesgue}
Let $(X,\U)$ be a compact uniform space and $\A$ be an open cover of $X$. Then there exist $p\in\U$, $\delta>0$ such that, for every $x\in X$, there exists $A\in\A$ such that 
\begin{equation*}
B_p(x,\delta) \subset A.
\end{equation*}
\end{lemma}

For more details concerning the uniform spaces, see e.g. \cite{Is,Ke,Pa}.

It will be also convenient to recall some elementary properties of multivalued maps which will be employed in the sequel. Let $\vp:X\multimap Y$ be a multivalued map and $X,Y$ be topological spaces. All multivalued maps will always have nonempty values, i.e. $\vp:X\to 2^X\setminus\{\emptyset\}$. If the space $Y$ is compact and all the values of $\vp$ are closed, then we can write $\vp:X\to\KK(Y)$, where
\begin{equation*}
\KK(Y):= \{A\subset Y: A\mbox{ is nonempty and compact}\}.
\end{equation*}

The regularity of semicontinuous multivalued maps can be defined by means of the preimages of $\vp\colon X\multimap Y$, where
\begin{equation*}
\vp^{-1}_-(B):= \{x\in X:\vp(x)\subset B\} \mbox{ (``small'' preimage of $\vp$ at $B\subset Y$),}
\end{equation*}
resp.
\begin{equation*}
\vp^{-1}_+(B):= \{x\in X:\vp(x)\cap B\neq\emptyset\} \mbox{ (``large'' preimage of $\vp$ at $B\subset Y$).}
\end{equation*}

\begin{definition}
\begin{enumerate}
	\item[(i)] $\vp:X\multimap Y$ is said to be \emph{upper semicontinuous} (u.s.c.) if $\vp^{-1}_-(B)$ is open for every open $B\subset Y$, resp. $\vp^{-1}_+(B)$ is closed for every closed $B\subset Y$;
	\item[(i)] $\vp:X\multimap Y$ is said to be \emph{lower semicontinuous} (l.s.c.) if $\vp^{-1}_+(B)$ is open for every open $B\subset Y$, resp. $\vp^{-1}_-(B)$ is closed for every closed $B\subset Y$;
	\item[(iii)] $\vp:X\multimap Y$ is said to be \emph{continuous} if it is both u.s.c. and l.s.c.
\end{enumerate}
\end{definition}

Obviously, if $\vp:X\to Y$ is single-valued u.s.c. or l.s.c., then it is continuous. 

\begin{lemma}[cf. e.g. {\cite[Proposition I.3.20]{AG}}, {\cite[Corollary 1.2.20]{HP}}]\label{l:2.2}
Let $X,Y$ be Hausdorff topological spaces and $\vp:X\to\KK(Y)$ be u.s.c. multivalued. If $K$ is a compact subset of $X$, then $\vp(K)$ is a compact subset of $Y$, i.e. if $K\in\KK(X)$, then $\bigcup_{x\in K} \vp(x)\in\KK(Y)$.
\end{lemma}


\begin{prop}[cf. {\cite[Corollary 1.2.22 and Proposition 1.2.33]{HP}}]\label{p:2.1}
A compact map $\vp:X\to\KK(Y)$ (i.e. $\vp(X)$ is contained in a compact subset of $Y$) is u.s.c. multivalued if and only if its graph $\Gamma_\vp:= \{(x,y)\in X\times Y: y\in \vp(x)\}$ is closed.
\end{prop}

Furthermore, if $\vp_1:X\multimap Y$ and $\vp_2:Y\multimap Z$ are u.s.c. multivalued (resp. l.s.c., continuous), then so is their composition $\vp_2\circ\vp_1:X\multimap Z$ (see e.g. \cite[Proposition I.2.56]{HP}).

If, in particular, $\vp_1:X\to\KK(Y)$ and $\vp_2:Y\to\KK(Z)$ are u.s.c. multivalued (resp. continuous), then so is $\vp_2\circ\vp_1:X\to\KK(Z)$ (see e.g. \cite[Proposition 3.21]{AG}, \cite[Proposition 1.2.56 and Corollary 1.2.20]{HP}).

If $X$ is a Hausdorff topological space and $Y$ is a metric space, then $\vp:X\to\KK(Y)$ is continuous if and only if it is continuous with respect to the Hausdorff metric $\dH$, where
\begin{equation*}
\dH(A,B):= \max\{\sup_{a\in A}(\inf_{b\in B} d(a,b),\sup_{b\in B}(\inf_{a\in A} d(a,b))\},
\end{equation*}
for $A,B\in\KK(Y)$ (see e.g. \cite[Theorem I.3.64]{AG}, \cite[Corollary 1.2.69]{HP}).

\begin{prop}[\cite{Mi51}]\label{p:hyper2}
Let $(X,\U)$ be a uniform space and $(X,\tau)$ be a Hausdorff topological space, where the topology $\tau$ is induced by the uniformity $\U$, as indicated above. Then $(\KK(X),\tv)$, where $\tv$ stands for a Vietoris topology, is a Hausdorff topological space induced by the uniformity $\U^{\mathrm{H}}$, where
\begin{equation*}
\U^{\mathrm{H}}:= \{\pH: p\in \U\mbox{ is bounded}\}
\end{equation*}
and 
\begin{equation*}
\pH(A,B) := \max \{ \sup_{a\in A} (\inf_{b\in B}p(a,b)), \sup_{b\in B} (\inf_{a\in A} p(a,b))\}, \mbox{ $A,B \in \KK(X)$, $p\in\U$.}
\end{equation*}

\end{prop}

\begin{prop}[cf. e.g. {\cite[Theorem 2.3.5]{KT2}}]\label{p:2.3}
Let $(X,\tau)$ be a topological space. Then $(X,\tau)$ is compact Hausdorff if and only if $(\KK(X),\tv)$ is compact Hausdorff, where $\tv$ stands for a Vietoris topology.
\end{prop}

\begin{remark}
Observe that although the single-valued hypermaps $\vp^*$, where $\vp^*(K):= \bigcup_{x\in K} \vp(x)$ for $K\in\KK(X)$, induced by the u.s.c. multivalued maps $\vp:X\to\KK(Y)$ in compact uniform (Hausdorff) spaces $X$, $Y$, send (in view of Propositions~\ref{p:hyper2} and \ref{p:2.3}) the values from compact Hausdorff spaces $(\KK(X),\tv(\KK(X)))$ into compact Hausdorff spaces $(\KK(Y),\tv(\KK(Y)))$, i.e. $\vp^*:\KK(X)\to\KK(Y)$, they can be discontinuous.
\end{remark}

On the other hand, if the maps $\vp:X\to\KK(Y)$ are continuous, then so are the induced hypermaps $\vp^*:\KK(X)\to\KK(Y)$, by means of the following proposition.

\begin{prop}[cf. {\cite[Proposition 5]{AL1}}]
Let $X=(X,\tau(X))$, $Y=(Y,\tau(Y))$ be compact topological spaces and $\vp:X\to\KK(Y)$ be a continuous multivalued map with compact values. Then the induced (single-valued) hypermap $\vp^*:\KK(X)\to\KK(Y)$, where $\vp^*(K):= \bigcup_{x\in K} \vp(x)$, for $K\in\KK(X)$, is continuous with respect to the Vietoris topologies $\tv(\KK(X))$ and $\tv(\KK(Y))$.

\end{prop}

For compact topological spaces $X=Y$, we can consider the sets of \emph{$n$-orbits} of $\vp$, namely
\begin{align*}
&\orb{1}(\vp):=X,\\
&\orbn(\vp):=\{(x_0,x_1,\ldots,x_{n-1})\in X^n: x_{i+1}\in\vp(x_i),i=0,1,\ldots,n-2\}, n\geq 2.
\end{align*}

For a sequence of multivalued maps $\vp_{0,\infty}:=\{\vp_j\}_{j=0}^{\infty}$, we can define the \emph{$n$-orbits} of $\vp_{0,\infty}$ as
\begin{align*}
&\orb{1}(\vp_{0,\infty}):=X,\\
&\orbn(\vp_{0,\infty}):= \{(x_0,x_1,\ldots,x_{n-1})\in X^n: x_{i+1}\in\vp_i(x_i),i=0,1,\ldots,n-2\}, n\geq 2,
\end{align*}
and in particular
\begin{align*}
\orbn(\vp_{0,\infty},x):= \{(x_0,x_1,\ldots,x_{n-1})\in\orbn(\vp_{0,\infty}),x_0=x\}, n\in\N,
\end{align*}
as those starting at $x\in X$.

%
%
%

\section{Topological entropy for arbitrary maps}\label{s:3}
For the correctness of some definitions of parametric topological entropy for arbitrary maps in compact Hausdorff spaces, it will be convenient to recall the one in \cite{AL4} (cf. the particular cases in \cite{KT,RT,Sh}), denoted by $\hKT$.

\begin{definition}\label{d:3.1}
Let $X$ be a nonempty set, $p$ a pseudometric on $X$ and $\ep>0$. A set $S\subset X$ is called \emph{$(p,\ep)$-separated} if $p(x,y) > \ep$, for every pair of distinct points $x,y\in S$. A set $R\subset X$ is called \emph{$(p,\ep)$-spanning in $Y\subset X$} if, for every $y\in Y$, there is $x\in R$ such that $p(x,y)\leq \ep$.  
\end{definition}

\begin{definition}\label{d:3.2}
Let $X$ be a nonempty set and $\vp_{0,\infty}$ be a sequence of multivalued maps $\vp_j:X\multimap X$, $j\in\N\cup\{0\}$. Let $p$ be a pseudometric on $X$, $\ep>0$ and $n\in\N$. Let $p_n$ be a pseudometric on $X^n$ defined as
\begin{align}\label{eq:0}
p_n((x_0,\ldots,x_{n-1}),(y_0,\ldots,y_{n-1})) := \max \{ p(x_i,y_i): i=0,\ldots,n-1\}.
\end{align}
We call $S\subset \orbn(\vp_{0,\infty})$ a \emph{$(p,\ep,n)_{\KT}$-separated set for $\vp_{0,\infty}$} if it is a $(p_n,\ep)$-separated subset. We call $R\subset \orbn(\vp_{0,\infty})$ a \emph{$(p,\ep,n)_{\KT}$-spanning set for $\vp_{0,\infty}$} if it is a $(p_n,\ep)$-spanning subset in $\orbn(\vp_{0,\infty})$.
\end{definition}

\begin{definition}\label{d:hsepspan}
Let $(X,\U)$ be a compact uniform space and $\vp_{0,\infty}$ be a sequence of multivalued maps $\vp_j:X\multimap X$, $j\in\N\cup\{0\}$. Denoting by $s_{\KT}(\vp_{0,\infty},p,\ep,n)$ the largest cardinality of a $(p,\ep,n)_{\KT}$-separated set for $\vp_{0,\infty}$, we take
\begin{equation*}
\hKT^{\sepp}(\vp_{0,\infty},p,\ep) := \limsup_{n\to\infty} \frac1{n}\log s_{\KT}(\vp_{0,\infty},p,\ep,n).
\end{equation*}
The \emph{topological entropy} $\hKT^{\sepp}(\vp_{0,\infty})$ of $\vp_{0,\infty}$ is defined to be
\begin{equation*}
\hKT^{\sepp}(\vp_{0,\infty}):= \sup_{p\in\U,\ep>0} \hKT^{\sepp}(\vp_{0,\infty},p,\ep).
\end{equation*}

Denoting by $r_{\KT}(\vp_{0,\infty},p,\ep,n)$ the smallest cardinality of a $(p,\ep,n)_{\KT}$-spanning set for $\vp_{0,\infty}$, we take
\begin{equation*}
\hKT^{\spa}(\vp_{0,\infty},p,\ep) := \limsup_{n\to\infty} \frac1{n}\log r_{\KT}(\vp_{0,\infty},p,\ep,n).
\end{equation*}
The \emph{topological entropy} $\hKT^{\spa}(\vp_{0,\infty})$ of $\vp_{0,\infty}$ is defined to be
\begin{equation*}
\hKT^{\spa}(\vp_{0,\infty}):= \sup_{p\in\U,\ep>0} \hKT^{\spa}(\vp_{0,\infty},p,\ep).
\end{equation*}
\end{definition}

\begin{remark}
The existence of the largest cardinality of a $(p,\ep,n)$-separated set for $\vp_{0,\infty}$ is, in view of Definition~\ref{d:hsepspan}, equivalent to the existence of the largest cardinality of a $(p_n,\ep)$-separated subset of $\orbn(\vp_{0,\infty})$. In \cite{AL4}, we have shown that these maximal cardinalities exist and both cardinalities $s_{\KT}(\vp_{0,\infty},p,\ep,n)$, $r_{\KT}(\vp_{0,\infty},p,\ep,n)$ are finite which guarantees the correctness of Definition~\ref{d:hsepspan}.
\end{remark}

Because of the equality $\hKT^{\sepp}(\vp_{0,\infty})=\hKT^{\spa}(\vp_{0,\infty})$ (see \cite[Theorem 3.1]{AL4}), we can simplify Definition~\ref{d:hsepspan} into the following form.

\begin{definition}\label{d:hKT}
Let $(X,\U)$ be a compact uniform space and $\vp_j: X\multimap X$, $j\in\N\cup\{0\}$, be a sequence $\vp_{0,\infty}$ of multivalued  maps. Then
\begin{equation*}
\hKT(\vp_{0,\infty}):= \hKT^{\sepp}(\vp_{0,\infty})=\hKT^{\spa}(\vp_{0,\infty}).
\end{equation*}
\end{definition}

\begin{lemma}\label{l:3.2}
Let $(X,\U)$ be a compact uniform space, $\vp_{0,\infty}$ be a sequence of arbitrary multivalued maps $\vp_j:X\multimap X$, $j\in\N\cup\{0\}$, and $\psi_{0,\infty}$ be a sequence of selections $\psi_j\subset \vp_j$ of $\vp_j$, i.e. $\psi_j(x)\subset\vp_j(x)$, $j\in\N\cup\{0\}$, for every $x\in X$. Then
\[
\hKT(\psi_{0,\infty}) \leq \hKT(\vp_{0,\infty}).
\]
\end{lemma}
\begin{proof}
For the proof, see \cite[Theorem 3.2 and Remark 3.5]{AL4}.
\end{proof}
The following definitions of topological entropy $\hCM^{\sepp}$ and $\hCM^{\spa}$ generalize in two directions those in \cite{CMM}, where autonomous multivalued maps were considered in compact metric spaces.

\begin{definition}\label{d:3.5}
Let $(X,\U)$ be a compact uniform space and $\vp_j:X\multimap X$, $j\in\N\cup\{0\}$, be a sequence $\vp_{0,\infty}$ of multivalued maps. For $p\in\U$, $n\in\N$, we define the pseudometric
\begin{align*}
p_n^{\CM}(x,y) := \inf_{\bar{x}\in\orbn(\vp_{0,\infty},x),\bar{y}\in\orbn(\vp_{0,\infty},y)} \max_{0\leq i \leq n-1} p(\bar{x}_i,\bar{y}_i), \quad x,y\in X,
\end{align*}
where $\bar{x}=(\bar{x}_0,\ldots,\bar{x}_{n-1})$, $\bar{y}=(\bar{y}_0,\ldots,\bar{y}_{n-1})$.

We call $S\subset X$ a \emph{$(p,\ep,n)_{\CM}$-separated set for $\vp_{0,\infty}$} if it is a $(p_n^{\CM},\ep)$-separated subset. We call $R\subset X$ a \emph{$(p,\ep,n)_{\CM}$-spanning set for $\vp_{0,\infty}$} if it is a $(p_n^{\CM},\ep)$-spanning subset in $X$.
\end{definition}

\begin{definition}\label{d:hCMsep}
Let $(X,\U)$ be a compact uniform space and $\vp_j: X\multimap X$, $j\in\N\cup\{0\}$, be a sequence $\vp_{0,\infty}$ of multivalued  maps. Denoting by $s_{\CM}(\vp_{0,\infty},p,\ep,n)$ the largest cardinality of a $(p,\ep,n)_{\CM}$-separated set for $\vp_{0,\infty}$, we take
\begin{equation*}
\hCM^{\sepp}(\vp_{0,\infty},p,\ep) := \limsup_{n\to\infty} \frac1{n}\log s_{\CM}(\vp_{0,\infty},p,\ep,n).
\end{equation*}
The \emph{topological entropy} $\hCM^{\sepp}(\vp_{0,\infty})$ of $\vp_{0,\infty}$ is defined to be
\begin{equation*}
\hCM^{\sepp}(\vp_{0,\infty}):= \sup_{p\in\U,\ep>0} \hCM^{\sepp}(\vp_{0,\infty},p,\ep).
\end{equation*}
\end{definition}

%
%

\begin{prop}\label{p:3.1}
Let $(X,\U)$ be a compact uniform space and $\vp_j:X\multimap X$, $j\in\N\cup\{0\}$, be a sequence $\vp_{0,\infty}$ of multivalued maps. Then the inequality
\[
\hCM^{\sepp}(\vp_{0,\infty}) \leq \hKT(\vp_{0,\infty})
\]
holds for the parametric topological entropies $\hKT(\vp_{0,\infty})$ and $\hCM^{\sepp}(\vp_{0,\infty})$ in the sense of Definitions~\ref{d:hsepspan} and~\ref{d:hCMsep}.
\end{prop}
\begin{proof}
Let $p\in\U$, $\ep>0$, $n\in\N$ and $S\subset X$ be a $(p,\ep,n)_{\CM}$-separated set for $\vp_{0,\infty}$. We define $\bar{S}\subset\orbn(\vp_{0,\infty})$ as a set of arbitrary extensions $\bar{x}\in\orbn(\vp_{0,\infty},x)$ of the elements $x\in S$, i.e. $\bar{S}=\{\bar{x}:x\in S\}$.

We show that $\bar{S}$ is a $(p,\ep,n)_{\KT}$-separated set for $\vp_{0,\infty}$. Consider $\bar{x},\bar{y}\in\bar{S}$. Then
\[
\max_{0\leq i \leq n-1} p(\bar{x}_i,\bar{y}_i)\geq p_n^{\CM}(\bar{x}_0,\bar{y}_0) = p_n^{\CM}(x,y).
\]

Since $\card S = \card \bar{S}$, we get $s_{\CM}(\vp_{0,\infty},p,\ep,n)\leq s_{\KT}(\vp_{0,\infty},p,\ep,n)$, for $\vp_{0,\infty}$, and so $\hCM^{\sepp}(\vp_{0,\infty})\leq \hKT^{\sepp}(\vp_{0,\infty})$.
\end{proof}

\begin{remark}
Definition~\ref{d:hCMsep} generalizes its analog in \cite[Definition 2.1]{CMM}, where $X$ was a compact metric space and $\vp_j=\vp$, for $j\in\N\cup\{0\}$.
\end{remark}

\begin{remark}
One can readily check from the proof of Proposition~\ref{p:3.1} that the number of cardinalities of $(p,\ep,n)_{\CM}$-separated sets for $\vp_{0,\infty}$ attains in Definition~\ref{d:hCMsep} really its finite maximum, which justifies the correctness of Definition~\ref{d:hCMsep}. This slightly improves the analogous particular definition in \cite[Definition 2.1]{CMM}, where suprema can be therefore replaced without any loss of generality by maxima. In \cite[Definition 3.3]{Ci}, maxima were correctly taken into account (in a single-valued context), but without any proof of correctness.
\end{remark}

\begin{definition}\label{d:hCMspa}
Let $(X,\U)$ be a compact uniform space and $\vp_{0,\infty}$ be a sequence of multivalued  maps $\vp_j: X\multimap X$, $j\in\N\cup\{0\}$. Denoting by $r_{\CM}(\vp_{0,\infty},p,\ep,n)$ the smallest cardinality of a $(p,\ep,n)_{\CM}$-spanning set for $\vp_{0,\infty}$, we take
\begin{equation*}
\hCM^{\spa}(\vp_{0,\infty},p,\ep) := \limsup_{n\to\infty} \frac1{n}\log r_{\CM}(\vp_{0,\infty},p,\ep,n).
\end{equation*}
The \emph{topological entropy} $\hCM^{\spa}(\vp_{0,\infty})$ of $\vp_{0,\infty}$ is defined to be
\begin{equation*}
\hCM^{\spa}(\vp_{0,\infty}):= \sup_{p\in\U,\ep>0} \hCM^{\spa}(\vp_{0,\infty},p,\ep).
\end{equation*}

\end{definition}

\begin{lemma}\label{l:hCMspasep}
Let $(X,\U)$ be a compact uniform space and $\vp_{0,\infty}$ be a sequence of multivalued  maps $\vp_j: X\multimap X$, $j\in\N\cup\{0\}$. Then the inequality
\[
\hCM^{\spa}(\vp_{0,\infty})\leq \hCM^{\sepp}(\vp_{0,\infty})
\]
holds for the parametric topological entropies $\hCM^{\spa}(\vp_{0,\infty})$ and $\hCM^{\sepp}(\vp_{0,\infty})$ of $\vp_{0,\infty}$, in the sense of Definitions~\ref{d:hCMsep} and~\ref{d:hCMspa}.
\end{lemma}
\begin{proof}
Let $p\in\U$, $\ep>0$, $n\in\N\cup\{0\}$. We will prove that $r_{\CM}(\vp_{0,\infty},p,\ep,n)\leq s_{\CM}(\vp_{0,\infty},p,\ep,n)$. Let $S\subset X$ be a $(p,\ep,n)_{\CM}$-separated set for $\vp_{0,\infty}$ with the maximal cardinality, i.e. $\card S = s_{\CM}(\vp_{0,\infty},p,\ep,n)$.

It suffices to show that $S$ is a $(p,\ep,n)_{\CM}$-spanning set in $X$ for $\vp_{0,\infty}$. If not, then there exists $y\in X$ such that $p_n^{\CM}(x,y)>\ep$, for every $x\in S$. Thus, $S\cup\{y\}$ is a $(p,\ep,n)_{\CM}$-separated set for $\vp_{0,\infty}$, which contradicts the assumption of a maximal cardinality.
\end{proof}

\begin{remark}\label{r:3.4}
The existence of a $(p,\ep,n)_{\CM}$-separated subset of $X$ for $\vp_{0,\infty}$ with the finite cardinality in Definition~\ref{d:hCMspa} follows directly from the arguments of the proof of Lemma~\ref{l:hCMspasep}. It justifies the correctness of Definition~\ref{d:hCMspa}.
\end{remark}

\begin{remark}
Definition~\ref{d:hCMspa} generalizes its analog in \cite[Definition 2.2]{CMM}, where $X$ was a compact metric space and $\vp_j=\vp$, for $j\in\N\cup\{0\}$.
\end{remark}

\begin{lemma}\label{l:hCMsel}
Let $X$ be a compact Hausdorff space, $\vp_{0,\infty}$ be a sequence of arbitrary multivalued maps $\vp_j:X\multimap X$, $j\in\N\cup\{0\}$, and $\psi_{0,\infty}$ be a sequence of selections $\psi_j\subset\vp_j$ of $\vp_j$, i.e. $\psi_j(x)\subset \vp_j(x)$, $j\in\N\cup\{0\}$, for every $x\in X$. Then  $\hCM^{\sepp}(\vp_{0,\infty})\leq\hCM^{\sepp}(\psi_{0,\infty})$ and $\hCM^{\spa}(\vp_{0,\infty})\leq\hCM^{\spa}(\psi_{0,\infty})$.
\end{lemma}

\begin{proof}
We can proceed in an analogous way as in the particular case in \cite[Theorem 3.2]{CMM}.

Let $p\in\U$, $\ep>0$, $n\in\N$. Assume $S\subset X$ is a $(p,\ep,n)_{\CM}$-separated set for $\vp_{0,\infty}$.

Take $x,y\in S$ and their arbitrary extensions $\bar{x}\in\orbn(\psi_{0,\infty},x)$, $\bar{y}\in\orbn(\psi_{0,\infty},y)$. Thanks to the inclusion $\orbn(\psi_{0,\infty},x)\subset\orbn(\vp_{0,\infty},x)$, $S$ is also a $(p,\ep,n)_{\CM}$-separated for $\psi_{\CM}$. Thus,
\[
s_{\CM}(\vp_{0,\infty},p,\ep,n)\leq s_{\CM}(\psi_{0,\infty},p,\ep,n),
\]
and subsequently $\hCM(\vp_{0,\infty})\leq \hCM(\psi_{0,\infty})$, as claimed.

Now, let $R\subset X$ be a $(p,\ep,n)_{\CM}$-spanning set for $\psi_{0,\infty}$. By the similar arguments as above, we can see that $R$ is a $(p,\ep,n)_{\CM}$-spanning set for $\vp_{0,\infty}$. Hence,
\[
r_{\CM}(\vp_{0,\infty},p,\ep,n)\leq r_{\CM}(\psi_{0,\infty},p,\ep,n),
\]
which, after accomplishing the limit processes from Definition~\ref{d:hCMspa}, completes the proof.
\end{proof}

For a particular case of u.s.c. multivalued maps with compact values, the following two definitions were already introduced in \cite{AL2}.

\begin{definition}\label{d:hrse}
Let $(X,\U)$ be a compact uniform space and $\vp_{0,\infty}$ be a sequence of multivalued  maps $\vp_j: X\multimap X$, $j\in\N\cup\{0\}$. A set $S\subset X$ is called \emph{$(p,\ep,n)_\rho$-separated} for $\vp_{0,\infty}$, for a positive integer $n\in\N$, $\ep>0$ and $p\in\U$, if for every pair of distinct points $x,y\in S$, $x\neq y$, there is at least one $k$ with $0\leq k < n$ such that
\[
p^\rho(\vp^{[k]}(x),\vp^{[k]}(y)) > \ep,
\]
where
\begin{align*}
p^\rho(A,B) &:= \inf \{ p(a,b): a\in A, b\in B\}, \mbox{ $A,B \subset X$,}\\
\vp^{[k]}&:= \vp_{k-1} \circ \ldots \circ \vp_0, \mbox{ for $k>0$, and } \vp^{[0]}:= \id_X.
\end{align*}

Let $s_{\rho}(\vp_{0,\infty},p,\ep,n)$ denote the maximal cardinality of a $(p,\ep,n)_{\rho}$-separated subset of $X$ for $\vp_{0,\infty}$, i.e.
\[
s_{\rho}(\vp_{0,\infty},p,\ep,n):= \max \{ \card S: \mbox{ $S\subset X$ is a $(p,\ep,n)_{\rho}$-separated set for $\vp$}\}.
\]

Then the \emph{topological entropy} $\hrse(\vp_{0,\infty})$ of $\vp_{0,\infty}$ is defined as
\begin{align*}\label{e:hrse}
\hrse(\vp_{0,\infty})&:= \sup_{p\in \U,\ep>0} s_{\rho}(\vp_{0,\infty},p,\ep),\\ 
\mbox{ where } s(\vp_{0,\infty},p,\ep)_{\rho}&:= \limsup_{n\to\infty} \frac1n \log s_{\rho}(\vp_{0,\infty},p,\ep,n).\nonumber
\end{align*}
\end{definition}

\begin{definition}\label{d:hrsp}
Let $(X,\U)$ be a compact uniform space and $\vp_{0,\infty}$ be a sequence of multivalued  maps $\vp_j: X\multimap X$, $j\in\N\cup\{0\}$. A set $R\subset X$ is called \emph{$(p,\ep,n)_{\rho}$-spanning} for $\vp_{0,\infty}$, for a positive integer $n\in\N$, $\ep>0$ and $p\in\U$, if for every $x\in X$ there is $y\in R$ such that (cf. Definition~\ref{d:hrse})
\[
p^\rho(\vp^{[k]}(x),\vp^{[k]}(y))\leq \ep \mbox{, for every $0\leq k <n$.}
\]

Let $r_{\rho}(\vp_{0,\infty},p,\ep,n)$ denote the least cardinality of a $(p,\ep,n)_{\rho}$-spanning subset of $X$ for $\vp_{0,\infty}$, i.e.
\[
r_{\rho}(\vp_{0,\infty},p,\ep,n):= \min \{ \card R: \mbox{ $R\subset X$ is a $(p,\ep,n)_{\rho}$-spanning set for $\vp$}\}.
\]

Then the \emph{topological entropy} $\hrsp(\vp_{0,\infty})$ of $\vp_{0,\infty}$ is defined as
\begin{align*}\label{e:hrsp}
\hrsp(\vp_{0,\infty})&:= \sup_{p\in \U,\ep>0} r_{\rho}(\vp_{0,\infty},p,\ep),\\ 
\mbox{ where } r_{\rho}(\vp_{0,\infty},p,\ep)&:= \limsup_{n\to\infty} \frac 1n \log r_{\rho}(\vp_{0,\infty},p,\ep,n).\nonumber
\end{align*}
\end{definition}

\begin{lemma}\label{l:3.5}
Let $X$ be a compact Hausdorff space and $\vp_{0,\infty}$ be a sequence of multivalued  maps $\vp_j: X\multimap X$, $j\in\N\cup\{0\}$. The inequality
\[
\hrsp(\vp_{0,\infty})\leq \hrse(\vp_{0,\infty})
\]
holds for the parametric topological entropies $\hrse(\vp_{0,\infty})$ and $\hrsp(\vp_{0,\infty})$ of $\vp_{0,\infty}$, in the sense of Definitions~\ref{d:hrse} and \ref{d:hrsp}.
\end{lemma}

\begin{proof}
It is sufficient to notice that a $(p,\ep,n)_\rho$-separating set with the maximal cardinality is a $(p,\ep,n)_\rho$-spanning set. Thus,
\begin{align*}
r_\rho(\vp_{0,\infty},p,\ep,n) \leq s_\rho(\vp_{0,\infty},p,\ep,n),
\end{align*}
and consequently, $\hrsp(\vp_{0,\infty})\leq \hrse(\vp_{0,\infty})$.
\end{proof}

\begin{prop}\label{p:3.2}
Let $(X,\U)$ be a compact uniform space and $\vp_{0,\infty}$ be a sequence of multivalued  maps $\vp_j: X\multimap X$, $j\in\N\cup\{0\}$. Then
\[
\hrse(\vp_{0,\infty})\leq \hCM^{\sepp}(\vp_{0,\infty}) \mbox{ and } \hrsp(\vp_{0,\infty})\leq \hCM^{\spa}(\vp_{0,\infty})
\]
hold for the parametric topological entropies $\hCM^{\sepp}(\vp_{0,\infty})$, $\hCM^{\spa}(\vp_{0,\infty})$ and $\hrse(\vp_{0,\infty})$, $\hrsp(\vp_{0,\infty})$ of $\vp_{0,\infty}$, in the sense of Definitions~\ref{d:hCMsep}-\ref{d:hrsp}.
\end{prop}
\begin{proof}
Let $p\in\U$, $\ep>0$ and $n\in\N$.

Taking, for $x,y\in X$,
\begin{align*}
p_n^{\rho}(x,y) &= \max_{i\in\{0,\ldots,n-1\}} \left(\inf_{\bar{x}\in\orbn(\vp_{0,\infty},x),\bar{y}\in\orbn(\vp_{0,\infty},y)} p(\bar{x}_i,\bar{y}_i)\right)\\ 
\mbox{ and (cf. Definition~\ref{d:3.5})}\\
p_n^{\CM}(x,y) &= \inf_{\bar{x}\in\orbn(\vp_{0,\infty},x),\bar{y}\in\orbn(\vp_{0,\infty},y)} \left(\max_{i\in\{0,\ldots,n-1\}} p(\bar{x}_i,\bar{y}_i)\right),
\end{align*}
we have, by means of the well-known principle (see e.g. \cite[Lemma 36.1]{Ro}),
\[
p_n^{\rho}(x,y) \leq p_n^{\CM}(x,y).
\]

Thus, every $(p,\ep,n)_{\CM}$-separated set is also $(p,\ep,n)_{\rho}$-separated, by which $s_{\rho}(\vp_{0,\infty},p,\ep,n) \leq s_{\CM}(\vp_{0,\infty},p,\ep,n)$, and so $\hrse(\vp_{0,\infty})\leq \hCM^{\sepp}(\vp_{0,\infty})$.

Similarly, every $(p,\ep,n)_{\rho}$-spanning set is also $(p,\ep,n)_{\CM}$-spanning. Consequently $r_{\rho}(\vp_{0,\infty},p,\ep,n) \leq r_{\CM}(\vp_{0,\infty},p,\ep,n)$, and so $\hrsp(\vp_{0,\infty})\leq \hCM^{\spa}(\vp_{0,\infty})$.

%
\end{proof}

\begin{remark}\label{r:3.6}
The existence of a $(p,\ep,n)_\rho$-separated subset of $X$ for $\vp_{0,\infty}$ with the maximal and finite cardinality in Definition~\ref{d:hrse} follows directly from the arguments of the proof of Proposition~\ref{p:3.2}. It justifies the correctness of Definition~\ref{d:hrse}.

Furthermore, the existence of a finite $(p,\ep,n)_\rho$-spanning subset of $X$ for $\vp_{0,\infty}$ in Definition~\ref{d:hrsp} follows directly from the arguments of the proof of Lemma~\ref{l:3.5}, which justifies the correctness of Definition~\ref{d:hrsp}.
\end{remark}

\begin{lemma}\label{l:3.6}
Let $X$ be a compact Hausdorff space, $\vp_{0,\infty}$ be a sequence of arbitrary multivalued maps $\vp_j:X\multimap X$, $j\in\N\cup\{0\}$, and $\psi_{0,\infty}$ be a sequence of selections $\psi_j\subset\vp_j$ of $\vp_j$, i.e. $\psi_j(x)\subset \vp_j(x)$, $j\in\N\cup\{0\}$, for every $x\in X$. Then the following inequalities
\begin{align*}
\hrse(\vp_{0,\infty}) \leq \hrse(\psi_{0,\infty}) \mbox{ and }\hrsp(\vp_{0,\infty}) \leq \hrsp(\psi_{0,\infty})
\end{align*}
hold for parametric topological entropies $\hrse(\vp_{0,\infty})$, $\hrse(\psi_{0,\infty})$ of $\vp_{0,\infty}$ and $\psi_{0,\infty}$, in the sense of Definition~\ref{d:hrse} and $\hrsp(\vp_{0,\infty})$, $\hrsp(\psi_{0,\infty})$ of $\vp_{0,\infty}$ and $\psi_{0,\infty}$, in the sense of Definition~\ref{d:hrsp}.
\end{lemma}

\begin{proof}
The statement can be done quite analogously as in \cite[Proposition 9]{AL2}.
\end{proof}

We can sum up the foregoing investigation as follows.
\begin{theorem}\label{t:3.1}
Let $X$ be a compact Hausdorff space, $\vp_{0,\infty}$ be a sequence of arbitrary multivalued maps $\vp_j:X\multimap X$, $j\in\N\cup\{0\}$, and $\psi_{0,\infty}$ be a sequence of selections $\psi_j\subset\vp_j$ of $\vp_j$, i.e. $\psi_j(x)\subset \vp_j(x)$, $j\in\N\cup\{0\}$, for every $x\in X$. Then the following inequalities
\begin{align*}
&\hrsp(\vp_{0,\infty}) \leq 
\begin{cases}
\hCM^{\spa}(\vp_{0,\infty})&\leq\\
\hrsp(\psi_{0,\infty})&\leq
\end{cases}
\begin{rrcases}
\hCM^{\sepp}(\vp_{0,\infty})\\
\hCM^{\spa}(\psi_{0,\infty})\\
\hrse(\psi_{0,\infty})
\end{rrcases}\\
&\leq \hCM^{\sepp}(\psi_{0,\infty})\leq\hKT(\psi_{0,\infty})\leq\hKT(\vp_{0,\infty})
\end{align*}
hold for the parametric topological entropies $\hKT(\vp_{0,\infty})$, $\hKT(\psi_{0,\infty})$ of $\vp_{0,\infty}$ and $\psi_{0,\infty}$, in the sense of Definition~\ref{d:hKT}, $\hCM^{\sepp}(\vp_{0,\infty})$, $\hCM^{\sepp}(\psi_{0,\infty})$ of $\vp_{0,\infty}$ and $\psi_{0,\infty}$, in the sense of Definition~\ref{d:hCMsep}, $\hCM^{\spa}(\vp_{0,\infty})$, $\hCM^{\spa}(\psi_{0,\infty})$ of $\vp_{0,\infty}$ and $\psi_{0,\infty}$, in the sense of Definition~\ref{d:hCMspa}, $\hrse(\vp_{0,\infty})$, $\hrse(\psi_{0,\infty})$ of $\vp_{0,\infty}$ and $\psi_{0,\infty}$, in the sense of Definition~\ref{d:hrse}, $\hrsp(\vp_{0,\infty})$, $\hrsp(\psi_{0,\infty})$ of $\vp_{0,\infty}$ and $\psi_{0,\infty}$, in the sense of Definition~\ref{d:hrsp}.
\end{theorem}

\begin{proof}
The statement follows directly from Propositions~\ref{p:3.1}, \ref{p:3.2} and Lemmas~\ref{l:3.2}-\ref{l:3.6}.
\end{proof}

\begin{remark}\label{r:3.7}
Since for a sequence $f_{0,\infty}$ of arbitrary single-valued maps $f_j:X\to X$, $j\in\N\cup\{0\}$, in a compact Hausdorff space $X$, all the definitions of topological entropy considered in Theorem~\ref{t:3.1} coincide, we can correctly define the parametric topological entropy $h(f_{0,\infty})$ of $f_{0,\infty}$ as
\begin{align}\label{eq:r:3.7}
h(f_{0,\infty})&:= \hKT(f_{0,\infty}) =\hCM^{\sepp}(f_{0,\infty})=\hCM^{\spa}(f_{0,\infty})\\
&= \hrse(f_{0,\infty})=\hrsp(f_{0,\infty}).\nonumber
\end{align}
\end{remark}

\section{Topological entropy of semicontinuous maps}
In this section, the multivalued maps under consideration will be at least semicontinuous. We start with l.s.c. maps.

The following two definitions will be introduced especially in order to be comparable with Definitions~\ref{d:hCMsep} and \ref{d:hCMspa}. Before their formulation, it will be convenient to present two auxiliary lemmas.

Hence, let $X$ be a compact topological space and $\vp_{0,\infty}$ be a sequence of l.s.c. multivalued maps $\vp_j:X\multimap X$, $j\in\N\cup\{0\}$. For $A_i\subset X$, $i=0,\ldots,n-1$, $n\in\N$, we define the set
\begin{align*}
\FF(\vp_{0,\infty};A_0,\ldots,A_{n-1}):= \pi_1(\orbn(\vp_{0,\infty})\cap (A_0\times\ldots \times A_{n-1}))\subset X,
\end{align*}
where $\pi_1$ denotes the projection of an $n$-orbit $\{x_i\}_{i=0}^{n-1}$ to its first component $x_0$, i.e. $\pi_1(\{x_i\}_{i=0}^{n-1})=x_0$.

In other words, it is a set of points $x_0\in X$ such that there exists an $n$-orbit $(x_0,\ldots,x_{n-1})\in\orbn(\vp_{0,\infty})$ with $x_i\in A_i$, $i=0,\ldots, n-1$.

\begin{lemma}\label{l:4.1}
Under the above assumptions, 
\begin{align*}
\FF(\vp_{0,\infty}; A_0)&=\pi_1(A_0),\\
\FF(\vp_{0,\infty}; A_0,\ldots,A_{n-1}) &= (\vp_0)^{-1}_+((\vp_1)^{-1}_+(\ldots(\vp_{n-2})^{-1}_+(A_{n-1})\cap A_{n-2}\ldots)\cap A_1)\cap A_0,
\end{align*}
 for $n>1$.
\end{lemma}
\begin{proof}
If $n=1$, then clearly $\FF(\vp_{0,\infty}; A_0)=\pi_1(X\cap A_0) = \pi_1(A_0)$. For $n>1$,
\begin{align*}
\FF(\vp_{0,\infty};A_0,\ldots,A_{n-1}) &= \pi_1(\orbn(\vp_{0,\infty})\cap(A_0\times \ldots \times A_{n-1}))\\
&= \{x_0\in A_0: \vp_0(x_0)\cap(\orb{n-1}(\psi_{0,\infty})\cap(A_1\times \ldots\times A_{n-1}))\}\\
&= A_0 \cap (\vp_0)^{-1}_{+}(\FF(\psi_{0,\infty};A_1,\ldots,A_{n-1})),
\end{align*}
where $\psi_j:=\vp_{j+1}$, for $j\in\N\cup\{0\}$.
\end{proof}

Defining, for an open cover $\A$ of $X$ and $n\in\N$, the systems of sets $\A^n$ and $\FF(\vp_{0,\infty};\A^n)$ as
\begin{align*}
\A^n &:= \{A_0\times\ldots\times A_{n-1}: A_0,\ldots, A_{n-1}\in\A\},\\
\FF(\vp_{0,\infty};\A^n) &:= \{\FF(\vp_{0,\infty};A_0,\ldots,A_{n-1}): A_0,\ldots, A_{n-1}\in\A\},
\end{align*}
we can state the following lemma.

\begin{lemma}\label{l:4.2}
Let $X$ be a compact topological space and $\vp_j: X\multimap X$, $j\in\N\cup\{0\}$, be a sequence $\vp_{0,\infty}$ of l.s.c. maps. If $\A$ is an open cover of $X$ and $n\in\N$, then $\FF(\vp_{0,\infty}; \A^n)$  is also an open cover of $X$.
\end{lemma}
\begin{proof}
If $A_0\times\ldots\times A_{n-1} \in\A^n$, then $\FF(\vp_{0,\infty};A_0,\ldots,A_{n-1})$ must be open by Lemma~\ref{l:4.1}, because the large preimage of an l.s.c. map is open and an intersection of a finite number of open sets is open as well.

Moreover, the system $\FF(\vp_{0,\infty}; \A^n)$ is a cover of $X$, because
\begin{align*}
\bigcup \FF(\vp_{0,\infty}; \A^n) &= \bigcup \{ \pi_1(\orbn(\vp_{0,\infty})\cap (A_0\times\ldots \times A_{n-1})): A_0\times\ldots\times A_{n-1}\in\A^n\}\\
&= \pi_1(\orbn(\vp_{0,\infty})\cap \bigcup \A^n) = \pi_1(\orbn(\vp_{0,\infty}\cap X^n) = X.
\end{align*}

\end{proof}

\begin{definition}\label{d:hU}
Let $X$ be a compact topological space and $\vp_j: X\multimap X$, $j\in\N\cup\{0\}$, be a sequence $\vp_{0,\infty}$ of l.s.c. maps. The \emph{topological entropy} $\hU(\vp,\A)$ of $\vp_{0,\infty}$ on the open cover $\A$ of $X$ takes the form
\begin{align*}
\hU(\vp_{0,\infty},\A) := \limsup_{n\to\infty} \frac1{n} \log N(\FF(\vp_{0,\infty},\A^n),
\end{align*}
where $N(\FF(\vp_{0,\infty};\A^n))$ stands for the minimal cardinality of a subcover of $\FF(\vp_{0,\infty};\A^n)$.

The \emph{topological (upper covering) entropy} $\hU(\vp_{0,\infty})$ of $\vp_{0,\infty}$ reads as
\begin{align*}
\hU(\vp_{0,\infty}) := \sup \{ \hU(\vp_{0,\infty},\A): \mbox{$\A$ is an open cover of $X$}\}.
\end{align*}
\end{definition}

\begin{remark}
The existence of a subcover with the minimal cardinality in Definition~\ref{d:hU} follows from the fact that $\FF(\vp_{0,\infty};\A^n)$ is an open cover of a compact topological space $X$ (see Lemma~\ref{l:4.2}), which justifies the correctness of Definition~\ref{d:hU}.
\end{remark}


\begin{prop}\label{p:4.1}
Let $(X,\U)$ be a compact uniform space and $\vp_{0,\infty}$ be a sequence of l.s.c. maps $\vp_j:X\multimap X$, $j\in\N\cup\{0\}$. Then the inequality
\[
\hU(\vp_{0,\infty}) \leq \hKT(\vp_{0,\infty})
\]
holds for the topological entropies $\hU(\vp_{0,\infty})$, $\hKT(\vp_{0,\infty})$ of $\vp_{0,\infty}$, in the sense of Definitions~\ref{d:hU} and \ref{d:hKT}.
\end{prop}

\begin{proof}
Let $\A$ be an open cover of $X$ and $n\in\N$. Assume that $\B$ is a subcover of $\A^n$ with the minimal cardinality among those covering $\orbn(\vp_{0,\infty})$. Then $\FF(\vp_{0,\infty};\B):=\{\FF(\vp_{0,\infty};B):B\in\B\}$ is, according to Lemma~\ref{l:4.2}, an open cover of $X$. Therefore, $N(\FF(\vp_{0,\infty};\B))\leq\card\B$, because $\card\FF(\vp_{0,\infty};\B)\leq\card\B$. Consequently $\hU(\vp_{0,\infty})\leq\hKT(\vp_{0,\infty})$ (see \cite[Definition 3.7 and Theorem 3.4]{AL4}).
\end{proof}

\begin{prop}\label{p:4.2}
Let $(X,\U)$ be a compact uniform space and $\vp_{0,\infty}$ be a sequence of l.s.c. maps $\vp_j:X\multimap X$, $j\in\N\cup\{0\}$. Then the inequality
\[
\hCM^{\sepp}(\vp_{0,\infty})\leq \hU(\vp_{0,\infty})
\]
holds for the topological entropies $\hCM^{\sepp}(\vp_{0,\infty})$, $\hU(\vp_{0,\infty})$ of $\vp_{0,\infty}$, in the sense of Definitions~\ref{d:hCMsep} and \ref{d:hU}.
\end{prop}

\begin{proof}
Let $\ep>0$, $p\in\U$ and $\A:=\{B_p(x,\frac{\ep}2): x\in X\}$ be an open cover of $X$. Let $n\in\N$ and $S\subset X$ be a $(p,\ep,n)_{\CM}$-separated subset of $X$ for $\vp_{0,\infty}$. Consider $A_0,\ldots,A_{n-1} \in\A$ and $x,y\in\FF(\vp_{0,\infty};A_0,\ldots,A_{n-1})\cap S$. There exist orbits $\bar{x}\in\orbn(\vp_{0,\infty},x)$, $\bar{y}\in\orbn(\vp_{0,\infty},y)$ such that $\bar{x}_i,\bar{y}_i\in A_i$, $i=0,\ldots,n-1$. Thus, $p(\bar{x}_i,\bar{y}_i)<\ep$, $i=0,\ldots,n-1$, and $p_n^{\CM}(x,y)<\ep$. By the separation property of $S$, $x=y$.

In this way, $\card S \leq N(\FF(\vp_{0,\infty};\A^n))$. Furthermore, $s_{\CM}(\vp_{0,\infty},p,\ep,n)\leq N(\FF(\vp_{0,\infty};\A^n))$, which already leads to the demanded inequality $\hCM^{\sepp}(\vp_{0,\infty})\leq \hU(\vp_{0,\infty})$.
\end{proof}

Let $X$ be a compact topological space and $\vp_j: X\multimap X$, $j\in\N\cup\{0\}$, be a sequence $\vp_{0,\infty}$ of l.s.c. maps.
For an open cover $\A$ of $X$, $x\in X$ and $n\in\N$, let us put
\begin{align*}
&\A^n(\vp_{0,\infty},x):= \\
&\pi_1\{ \bar{y}\in\orbn(\vp_{0,\infty}): \exists \bar{x}\in\orbn(\vp_{0,\infty},x) \exists A_i\in\A , \bar{x}_i,\bar{y}_i\in A_i, i=0,\ldots,n-1\},\\
&\A^n(\vp_{0,\infty}):= \{\A^n(\vp_{0,\infty},x): x\in X\}.
\end{align*}

\begin{lemma}
Let $X$ be a compact topological space and $\vp_j: X\multimap X$, $j\in\N\cup\{0\}$, be a sequence $\vp_{0,\infty}$ of l.s.c. maps.
For an open cover $\A$ of $X$, $x\in X$ and $n\in\N$, the set $\A^n(\vp_{0,\infty},x)$ is nonempty and open, and $\A^n(\vp_{0,\infty})$ is an open cover of $X$.
\end{lemma}
\begin{proof}
For $x\in X$, $\A^n(\vp_{0,\infty},x)$ is an open set by means of Lemma~\ref{l:4.2}, jointly with the equality
\[
\A^n(\vp_{0,\infty},x) = \bigcup_{(A_0\times\ldots\times A_{n-1})\in\A^{n},\, x\in\FF(\vp_{0,\infty};A_0,\ldots,A_{n-1})} \FF(\vp_{0,\infty};A_0,\ldots,A_{n-1}).
\]

We deduce directly from the definition that $x\in\A^n(\vp_{0,\infty},x)$, by which \\$\A^n(\vp_{0,\infty},x)\neq\emptyset$, and at the same time $\A^n(\vp_{0,\infty})$ covers $X$.
\end{proof}
\begin{definition}\label{d:hL}
Let $X$ be a compact topological space and $\vp_j: X\multimap X$, $j\in\N\cup\{0\}$, be a sequence $\vp_{0,\infty}$ of l.s.c. maps. The \emph{topological entropy} $\hL(\vp,\A)$ of $\vp_{0,\infty}$ on the open cover $\A$ of $X$ takes the form
\begin{align*}
\hL(\vp_{0,\infty},\A) := \limsup_{n\to\infty} \frac1{n} \log N(\A^n(\vp_{0,\infty})),
\end{align*}
where $N(\A^n(\vp_{0,\infty}))$ stands for the minimal cardinality of a subcover of $\A^n(\vp_{0,\infty})$.

The \emph{topological (lower covering) entropy} $\hL(\vp_{0,\infty})$ of $\vp_{0,\infty}$ reads as
\begin{align*}
\hL(\vp_{0,\infty}) := \sup \{ \hL(\vp_{0,\infty},\A): \mbox{$\A$ is an open cover of $X$}\}.
\end{align*}

\end{definition}

\begin{prop}\label{p:4.3}
Let $(X,\U)$ be a compact uniform space and $\vp_{0,\infty}$ be a sequence of l.s.c. maps $\vp_j:X\multimap X$, $j\in\N\cup\{0\}$. Then the inequality
\[
\hL(\vp_{0,\infty})=\hCM^{\spa}(\vp_{0,\infty})
\]
holds for the topological entropies $\hL(\vp_{0,\infty})$, $\hCM^{\spa}(\vp_{0,\infty})$ of $\vp_{0,\infty}$, in the sense of Definitions~\ref{d:hL} and \ref{d:hCMspa}.
\end{prop}
\begin{proof}
At first, we will prove the inequality $\hL(\vp_{0,\infty})\leq \hCM^{\spa}(\vp_{0,\infty})$. Let $\A$ be an open cover of $X$ with $p\in\U$, $\delta>0$, provided by Lemma~\ref{l:Lebesgue}. Let $R\subset X$ be an $(p,\frac{\delta}2,n)_{\CM}$-spanning set in $X$ for $\vp_{0,\infty}$.

For $x\in R$, we take $y\in X$ such that $p_n^{\CM}(x,y)\leq\frac{\delta}2$ (cf. Definition~\ref{d:hCMspa}). Then there exist $\bar{x},\bar{y}\in\orbn(\vp_{0,\infty})$ such that $p(\bar{x}_i,\bar{y}_i))<\delta$, i.e. $\bar{y}_i\in B_p(\bar{x}_i,\delta)$, for every $i=0,\ldots,n-1$. Furthermore, according to Lemma~\ref{l:Lebesgue}, there exist $A_i\in\A$ such that $B_p(\bar{x}_i,\delta)\subset A_i$, for every $i=0,\ldots,n-1$. Thus, $\bar{x}_i,\bar{y}_i\in A_i$, for every $i=0,\ldots,n-1$, and $y\in\A^n(\vp_{0,\infty},x)$. Since $\{\A^n(\vp_{0,\infty},x):x\in R\}$ is an open subcover of $\A^n(\vp_{0,\infty})$, we get $N(\A^n(\vp_{0,\infty}))\leq \card\{ \A^n(\vp_{0,\infty},x): x\in R\} \leq \card R$.

Moreover, $N(\A^n(\vp_{0,\infty}))\leq r_{\CM}(\vp_{0,\infty},p,\frac{\delta}2,n)$. Passing to the limit for $n\to\infty$ and taking suprema with respect to $p\in\U$ and $\delta>0$, we arrive at $\hL(\vp_{0,\infty})\leq \hCM^{\spa}(\vp_{0,\infty})$, as claimed.

Now, we will prove the reverse inequality $\hL(\vp_{0,\infty})\geq \hCM^{\spa}(\vp_{0,\infty})$. Let $\delta>0$, $p\in\U$ and $\A:=\{B_p(x,\frac{\delta}2):x\in X\}$.

For $n\in\N$, we take the subcover $\B$ of $\A^n(\vp_{0,\infty})$ with a (finite) minimal cardinality. Then there exists $R\subset X$ such that
\[
\B = \{\A^n(\vp_{0,\infty},x): x\in R\},
\]
$\card R = N(\A^n(\vp_{0,\infty}))$ and $\bigcup_{x\in R} \A^n(\vp_{0,\infty},x)=X$.

Letting $x\in R$ and $y\in\A^n(\vp_{0,\infty},x)$, there exist $\bar{x},\bar{y}\in\orbn(\vp_{0,\infty})$ such that $\bar{x}_i,\bar{y}_i\in A_i$ where $A_i\in\A$, $i=0,\ldots,n-1$. Since $p(\bar{x}_i,\bar{y}_i)<\delta$, for all $i=0,\ldots,n-1$, in view of Definition~\ref{d:3.5}, it holds
\[
p_n^{\CM}(x,y)<\delta.
\]

For every $y\in X$, there exists $x\in R$ such that $y\in\A^n(\vp_{0,\infty},x)$. Therefore, $p_n^{\CM}(x,y)<\delta$, and $R$ is a $(p,\delta,n)_{\CM}$-spanning in $X$ for $\vp_{0,\infty}$, by which
\[
r_{\CM}(\vp_{0,\infty},p,\delta,n)\leq N(\A^n(\vp_{0,\infty})).
\]

After all, 
\begin{align*}
\limsup_{n\to\infty} \frac1{n} \log r_{\CM}(\vp_{0,\infty},p,\delta,n) &\leq \limsup_{n\to\infty} \frac1{n} \log N(\A^n(\vp_{0,\infty}))\\
& = \hL(\vp_{0,\infty},\A) \leq \hL(\vp_{0,\infty}).
\end{align*}
Taking suprema over $\delta>0$, $p\in\U$, we arrive at the desired inequality
\[
\hCM^{\spa}(\vp_{0,\infty})\leq \hL(\vp_{0,\infty}).
\]

\end{proof}

Summing up, we can formulate the following theorem for l.s.c. maps.

\begin{theorem}\label{t:4.1}
Let $X$ be a compact Hausdorff space and $\vp_{0,\infty}$ be a sequence of l.s.c. multivalued maps $\vp_j:X\multimap X$, $j\in\N\cup\{0\}$. Then the following relations
\begin{align*}
\hL(\vp_{0,\infty}) = \hCM^{\spa}(\vp_{0,\infty}) \leq \hCM^{\sepp}(\vp_{0,\infty}) \leq \hU(\vp_{0,\infty}) \leq \hKT(\vp_{0,\infty})
\end{align*}
hold for the parametric topological entropies of $\vp_{0,\infty}$, in the sense of Definitions~\ref{d:hL}, \ref{d:hCMspa}, \ref{d:hCMsep}, \ref{d:hU} and \ref{d:hKT}, respectively.
\end{theorem}

\begin{proof}
The statement follows directly from Lemma~\ref{l:hCMspasep} and Propositions~\ref{p:4.1}-\ref{p:4.3}.
\end{proof}

Now, we will turn to u.s.c. multivalued maps. As already documented in \cite[Remark 5.4]{AL4}, there is practically no need to verify the correctness of definitions of parametric topological entropy of u.s.c. maps with compact values.

The following definition of parametric branch entropy generalizes its analog e.g. in \cite{Hu}, where  it was defined for the first time for the inversions of autonomous single-valued maps and in \cite[Definition 2.3]{WZZ} for autonomous u.s.c. maps in compact metric spaces.

Taking the pseudometric $p_n^b$ on $X$ as
\begin{align}\label{eq:pnb}
p^b_n(x,y)&:= \pnH(\orbn(\vp_{0,\infty},x),\orbn(\vp_{0,\infty},y))\\
&=\max\{\sup_{\overline{x}\in\orbn(\vp_{0,\infty},x)} \inf_{\overline{y}\in\orbn(\vp_{0,\infty},y)} p_n(\overline{x},\overline{y}), \nonumber\\ 
&\sup_{\overline{y}\in\orbn(\vp_{0,\infty},y)} \inf_{\overline{x}\in\orbn(\vp_{0,\infty},x)} p_n(\overline{y},\overline{x})\},\nonumber
\end{align}
where $p_n\in\U^n$ is defined by \eqref{eq:0} and $\pnH$ along the lines of Proposition~\ref{p:hyper2}, we can generalize the branch entropy as follows.

\begin{definition}\label{d:hi}
Let $(X,\U)$ be a compact uniform space and $\vp_j:X\to\KK(X)$, $j\in\N\cup\{0\}$, be a sequence $\vp_{0,\infty}$ of u.s.c. multivalued maps. Denoting by $s_{\mathfrak{i}}(\vp_{0,\infty},p^b_n,\ep)$ the maximum of cardinalities of $(p^b_n,\ep)$-separated sets of $X$, resp. by $r_{\mathfrak{i}}(\vp_{0,\infty},p^b_n,\ep)$ the smallest cardinality of a $(p^b_n,\ep)$-spanning set in $X$, the \emph{parametric branch entropy} $\hi(\vp_{0,\infty})$ is defined to be
\begin{align*}
\hi(\vp_{0,\infty})&:= \sup_{p\in\U,\ep>0} \limsup_{n\to\infty} \frac1{n} \log s_{\mathfrak{i}}(\vp_{0,\infty},p^b_n,\ep)\\
&= \sup_{p\in\U,\ep>0} \limsup_{n\to\infty} \frac1{n} \log r_{\mathfrak{i}}(\vp_{0,\infty},p^b_n,\ep),
\end{align*}
\end{definition}

Let us also recall the following definitions in \cite[Definitions 7 and 8]{AL2}, which will play an important role in the next section.

\begin{definition}[Via separated sets]\label{d:hHse}
Let $(X,\U)$ be a compact uniform space and $\vp_j\colon X\to\KK(X)$, $j\in\N\cup\{0\}$, be a sequence $\vp_{0,\infty}$ of u.s.c. multivalued  maps. A set $S\subset X$ is called \emph{$(p,\ep,n)$-separated} for $\vp_{0,\infty}$, for a positive integer $n\in\N$, $\ep>0$ and $p\in\U$, if for every pair of distinct points $x,y\in S$, $x\neq y$, there is at least one $k$ with $0\leq k < n$ such that
\[
p^{\H}(\vp^{[k]}(x),\vp^{[k]}(y)) > \ep,
\]
where (cf. Proposition~\ref{p:hyper2}) 
\begin{align*}
p^{\H}(A,B) &:= \max \{ \sup_{a\in A} p(a,B), \sup_{b\in B} p(A,b)\}, \mbox{ $A,B \in \KK(X)$,}\\
\vp^{[k]}&:= \vp_{k-1} \circ \ldots \circ \vp_0, \mbox{ for $k>0$, and } \vp^{[0]}:= \id_X.
\end{align*}
Let $s_{\H}(\vp_{0,\infty},p,\ep,n)$ denote the largest cardinality of a $(p,\ep,n)$-separated subset of $X$ with respect to $\vp_{0,\infty}$, i.e.
\[
s_{\H}(\vp_{0,\infty},p,\ep,n):= \max \{ \card S: \mbox{ $S\subset X$ is a $(p,\ep,n)$-separated set for $\vp$}\}.
\]

Then the \emph{topological entropy} $\hHse(\vp_{0,\infty})$ of $\vp_{0,\infty}$ is defined as
\begin{align*}\label{e:hHse}
\hHse(\vp_{0,\infty})&:= \sup_{p\in \U,\ep>0} s_{\H}(\vp_{0,\infty},p,\ep),\\ 
\mbox{ where } s_{\H}(\vp_{0,\infty},p,\ep)&:= \limsup_{n\to\infty} \frac1n \log s_{\H}(\vp_{0,\infty},p,\ep,n).\nonumber
\end{align*}
\end{definition}

\begin{remark}\label{r:5}
Even for $\vp_j=\vp$, $j\in\N\cup\{0\}$, Definition~\ref{d:hHse} is new. In compact metric spaces, the nonparametric version of Definition~\ref{d:hHse} reduces to \cite[Definition 9]{AL1}. For single-valued maps $\vp_j\colon X\to X$, $j\in\N\cup\{0\}$, the corresponding definition to Definition \ref{d:hHse} was given in \cite[Section 2.2]{Sh}. This correspondence might not be, for the first glance, evident because of different (though equivalent) definitions of a uniform structure involved in Definition \ref{d:hHse} and the one in \cite[Section 2.2]{Sh}.
\end{remark}

\begin{definition}[Via spanning sets]\label{d:hHsp}
Let $(X,\U)$ be a compact uniform space and $\vp_j\colon X\to\KK(X)$, $j\in\N\cup\{0\}$, be a sequence $\vp_{0,\infty}$ of u.s.c. multivalued  maps. A set $R\subset X$ is called \emph{$(p,\ep,n)$-spanning} for $\vp_{0,\infty}$, for a positive integer $n\in\N$, $\ep>0$ and $p\in\U$, if for every $x\in X$ there is $y\in R$ such that (cf. Definition~\ref{d:hHse})
\[
p^H(\vp^{[k]}(x),\vp^{[k]}(y))\leq \ep \mbox{, for every $0\leq k <n$}.
\]

Let $r_{\H}(\vp_{0,\infty},p,\ep,n)$ denote the least cardinality of a $(p,\ep,n)$-spanning subset of $X$ with respect to $\vp_{0,\infty}$, i.e.
\[
r_{\H}(\vp_{0,\infty},p,\ep,n):= \min \{ \card R: \mbox{ $R\subset X$ is a $(p,\ep,n)$-spanning set for $\vp$}\}.
\]

Then the \emph{topological entropy} $\hHsp(\vp_{0,\infty})$ of $\vp_{0,\infty}$ is defined as
\begin{align*}\label{e:hHsp}
\hHsp(\vp_{0,\infty})&:= \sup_{p\in \U,\ep>0} r_{\H}(\vp_{0,\infty},p,\ep),\\
\mbox{ where } r_{\H}(\vp_{0,\infty},p,\ep)&:= \limsup_{n\to\infty} \frac1n \log r_{\H}(\vp_{0,\infty},p,\ep,n).\nonumber
\end{align*}
\end{definition}

Because of the equality (cf. \cite[Proposition 8]{AL2})
\[
\hH^{\sepp}(\vp_{0,\infty}) = \hH^{\spa}(\vp_{0,\infty}),
\]
we can define the parametric topological entropy $\hH(\vp_{0,\infty})$ for u.s.c. multivalued maps $\vp_j:X\to\KK(X)$, $j\in\N\cup\{0\}$, in a compact Hausdorff space as follows.

\begin{definition}\label{d:hH}
Let $(X,\U)$ be a compact uniform space and $\vp_j\colon X\to\KK(X)$, $j\in\N\cup\{0\}$, be a sequence $\vp_{0,\infty}$ of u.s.c. multivalued  maps. Then
\[
\hH(\vp_{0,\infty}) := \hH^{\sepp}(\vp_{0,\infty}) =\hH^{\spa}(\vp_{0,\infty}).
\]
\end{definition}

\begin{prop}\label{p:4.4}
Let $X$ be a compact Hausdorff space and $\vp_j:X\to\KK(X)$, $j\in\N\cup \{0\}$, be a sequence $\vp_{0, \infty}$ of u.s.c. multivalued maps. The inequality
 \begin{align}\label{eq:p:4.4}
\hH(\vp_{0,\infty}) \leq \hi(\vp_{0,\infty})
 \end{align}
holds for the parametric topological entropies $\hH(\vp_{0,\infty})$, $\hi(\vp_{0,\infty})$ of $\vp_{0,\infty}$, in the sense of Definitions~\ref{d:hH} and \ref{d:hi}.
\end{prop}
\begin{proof}
Let $p\in\U$, $n\in\N$ and $x,y\in X$. Then, applying \cite[Lemma 36.1]{Ro} (see also \eqref{eq:pnb}), we obtain
\begin{align*}
p_n^b(x,y) &= p_n^{\H}(\orbn(x),\orbn(y)) \\
 &= \max\{ \max_{\bar{x}\in\orbn(\vp_{0,\infty},x)} \min_{\bar{y}\in\orbn(\vp_{0,\infty},y)} \max_{0\leq k<n} p(\bar{x}_k,\bar{y}_k),\\
 &\max_{\bar{y}\in\orbn(\vp_{0,\infty},y)} \min_{\bar{x}\in\orbn(\vp_{0,\infty},x)} \max_{0\leq k<n} p(\bar{x}_k,\bar{y}_k)\}\\
 &\geq \max\{ \max_{0\leq k<n} \max_{\bar{x}\in\orbn(\vp_{0,\infty},x)} \min_{\bar{y}\in\orbn(\vp_{0,\infty},y)}  p(\bar{x}_k,\bar{y}_k),\\
 &\max_{0\leq k<n}\max_{\bar{y}\in\orbn(\vp_{0,\infty},y)} \min_{\bar{x}\in\orbn(\vp_{0,\infty},x)}  p(\bar{x}_k,\bar{y}_k)\}\\
&= \max_{0\leq k<n} \max\{\max_{\bar{x}\in\orbn(\vp_{0,\infty},x)} \min_{\bar{y}\in\orbn(\vp_{0,\infty},y)}  p(\bar{x}_k,\bar{y}_k),\\
 &\max_{\bar{y}\in\orbn(\vp_{0,\infty},y)} \min_{\bar{x}\in\orbn(\vp_{0,\infty},x)}  p(\bar{x}_k,\bar{y}_k)\}\\
&= \max_{0\leq k <n} \pH(\vp^{[k]}(x),\vp^{[k]}(y)).
\end{align*}
Thus, a usual inspection of Definitions~\ref{d:hH} and~\ref{d:hi} leads to the conclusion \eqref{eq:p:4.4}.

\end{proof}

Summing up, we can formulate the following theorem for u.s.c. maps.

\begin{theorem}\label{t:4.2}
Let $X$ be a compact Hausdorff space, $\vp_{0,\infty}$ be a sequence of u.s.c. multivalued maps $\vp_j:X\to\KK(X)$, $j\in\N\cup\{0\}$. Then the following inequalities
 \[
\hrsp(\vp_{0,\infty}) \leq \hrse(\vp_{0,\infty}) \leq \hH(\vp_{0,\infty})\leq \hi(\vp_{0,\infty})
 \]
hold for the parametric topological entropies of $\vp_{0,\infty}$, in the sense of Definitions~\ref{d:hrsp}, \ref{d:hrse}, \ref{d:hH}, \ref{d:hi}, respectively.
\end{theorem}
\begin{proof}
The statement follows directly from Lemma~\ref{l:3.5}, Proposition~\ref{p:4.4} and \cite[Theorem 3]{AL2}.
\end{proof}

\begin{remark}\label{r:4.3}
Since for a sequence $f_{0,\infty}$ of continuous single-valued maps $f_j:X\to X$, $j\in\N\cup\{0\}$, in a compact Hausdorff space $X$, all the definitions of topological entropy considered in Theorems~\ref{t:3.1}, \ref{t:4.1} and \ref{t:4.2} coincide, we can define again the parametric topological entropy $h(f_{0,\infty})$ of $f_{0,\infty}$, this time, as
\begin{align*}\label{eq:r:4.3}
h(f_{0,\infty})&:= \hKT(f_{0,\infty}) = \hU(f_{0,\infty})=\hCM^{\sepp}(f_{0,\infty}) = \hCM^{\spa}(f_{0,\infty})\\
&=\hL(f_{0,\infty}) = \hi(f_{0,\infty}) = \hH(f_{0,\infty}) = \hrse(f_{0,\infty}) = \hrsp(f_{0,\infty}). \nonumber
\end{align*}

In this way, such a definition of $h(f_{0,\infty})$ is equivalent with those in \cite[Section 2]{Sh} and \cite[Section 1]{KS}, where, in the latter case, the compact topological space $X$ can be even quite arbitrary.
\end{remark}

\section{Topological entropy for discontinuous induced hypermaps}
In this section, we will investigate for the first time the relationship between topological entropy of u.s.c. maps with compact values in a compact Hausdorff space $X$ and the one of the induced possibly discontinuous hypermaps in the corresponding hyperspace $\KK(X)$. This is possible thanks to Lemma~\ref{l:2.2} and Propositions~\ref{p:hyper2} and~\ref{p:2.3}, by which we have a single-valued $\vp^*:\KK(X)\to\KK(X)$, where $\vp^*(K):=\bigcup_{x\in K}\vp(x)$, for $K\in\KK(X)$, in the compact Hausdorff hyperspace $\KK(X)$, endowed with the Vietoris topology.

Hence, for the sequence $\vp_{0,\infty}$ of multivalued u.s.c. maps $\vp_j:X\to\KK(X)$, $j\in\N\cup\{0\}$, in a compact Hausdorff space $X$, we can consider the related sequence $\vp^*_{0,\infty}:=\{\vp^*_j\}_{j=0}^\infty$, where $(\vp^*)^{[k]}:=\vp_{k-1}^*\circ\ldots\circ\vp^*_0$, for all $k\in\N$, and $(\vp^*)^{[0]}:=\id\restr{\KK(X)}$.

The crucial related inequality reads as follows.
\begin{prop}\label{p:5.1}
Let $(X,\U)$ be a compact uniform space and $\vp_{0,\infty}$ be a sequence of u.s.c. maps $\vp_j\colon X\to\KK(X)$, $j\in\N\cup\{0\}$. Then the inequality
\[
\hH(\vp_{0,\infty}) \leq h(\vp_{0,\infty}^*)
\]
holds for the topological entropies $\hH(\vp_{0,\infty})$ of $\vp_{0,\infty}$ in the sense of Definition~\ref{d:hH} and $h(\vp^*_{0,\infty})$ of the sequence $\vp_{0,\infty}^*$ of the induced hypermaps $\vp^*_j:\KK(X)\to\KK(X)$, $j\in\N\cup\{0\}$, which can be defined in any way by means of the equalities \eqref{eq:r:3.7} in Remark~\ref{r:3.7}.
\end{prop}

\begin{proof}
Since $(X,\U)$ is a compact uniform space, $(\KK(X),\U^{H})$ is also a compact uniform space (see Proposition~\ref{p:hyper2}). Let $p\in\U$ and let $S\subset X$ be a $(p,\ep,n)_{H}$-separated for $\vp_{0,\infty}$, i.e. for each $x,y\in S$, $x\neq y$,
\[
\max_{0\leq k< n} p^{\H}(\vp^{[k]}(x),\vp^{[k]}(y)) > \ep.
\]

One can simply verify that, for every $z\in X$,
\[
\vp^{[k]}(z) = (\vp^*)^{[k]}(i(z)),
\]
where $i:X\to\KK(X)$ is a natural inclusion. Therefore,
\[
\max_{0\leq k< n} p^{\H}\left((\vp_{0,\infty}^*)^{[k]}(i(x)),(\vp_{0,\infty}^*)^{[k]}(i(y))\right) > \ep,
\]
and $i(E)\subset\KK(X)$ is a $(p^{\H},\ep,n)$-separated set. Using the standard arguments, we can conclude that $h(\vp_{0,\infty}^*) \geq h_{\H}(\vp_{0,\infty})$.
\end{proof}

Summing up Theorem~\ref{t:4.2} and Proposition~\ref{p:5.1}, we arrive at the fourth main theorem.
\begin{theorem}\label{t:5.2}
Let $(X,\U)$ be a compact uniform space and $\vp_{0,\infty}$ be a sequence of u.s.c. maps $\vp_j\colon X\to\KK(X)$, $j\in\N\cup\{0\}$. Then the following inequalities
\begin{align*}
\hrsp(\vp_{0,\infty}) \leq \hrse(\vp_{0,\infty}) \leq \hH(\vp_{0,\infty}) \leq h(\vp_{0,\infty}^*)
\end{align*}
hold for the parametric topological entropies $\hrsp(\vp_{0,\infty})$, $\hrse(\vp_{0,\infty})$, $\hH(\vp_{0,\infty})$ of $\vp_{0,\infty}$, in the sense of Definitions~\ref{d:hrsp}, \ref{d:hrse} and \ref{d:hH}, respectively, and $h(\vp^*_{0,\infty})$ of $\vp^*_{0,\infty}$, with regard to \eqref{eq:r:3.7} in Remark~\ref{r:3.7}.
\end{theorem}

\begin{remark}\label{r:5.1}
In \cite[Example 3]{AL1}, we have shown an example indicating that the inequality $\hKT(\vp_{0,\infty}) \leq h(\vp^*_{0,\infty})$ does not hold for a certain continuous multivalued map $\vp=\vp_j$, $j\in\N\cup\{0\}$. As concerns the validity of the inequality $\hi(\vp_{0,\infty})\leq h(\vp^*_{0,\infty})$, for the branch entropy $\hi(\vp_{0,\infty})$ of $\vp_{0,\infty}$, it is an open problem.
\end{remark}

\section{Some further possibilities for particular subclasses of u.s.c. maps}
In this section, we will show that the desired inequality $h(\vp_{0,\infty})\leq h(\vp^*_{0,\infty})$ can be also satisfied, under certain additional restrictions imposed on suitable subclasses of u.s.c. maps, for some further definitions of topological entropy than those treated in the foregoing section.

In reply to an open problem posed in our paper \cite{AL2}, namely whether or not the inequality $\hCM^{\sepp}(\vp)\leq h(\vp^*)$ holds for at least continuous maps $\vp$, we are able to answer it only in a particular way by means of the following proposition.

\begin{prop}\label{p:6.1}
Let $(X,\U)$ be a compact uniform space and $\vp_{0,\infty}$ be a sequence of u.s.c. maps $\vp_j\colon X\to\KK(X)$, $j\in\N\cup\{0\}$. If there exist single-valued selections $f_j\subset \vp_j$, i.e. $f_j(x)\in\vp_j(x)$, $j\in\N\cup\{0\}$, for every $x\in X$, such that
\begin{align}\label{eq:p:6.1a}
p^{\H}(\vp^{[j]}(x),\vp^{[j]}(y)) = p(f^{[j]}(x),f^{[j]}(y)),
\end{align}
holds for every $j\in\N\cup\{0\}$, $p\in\U$ and all $x,y\in X\setminus A$, where $A\subset X$ is a finite (possibly empty) subset. Then the relations
\begin{align}\label{eq:p:6.1b}
\hCM^{\spa}(\vp_{0,\infty}) \leq \hCM^{\sepp}(\vp_{0,\infty})\leq h(f_{0,\infty}) = \hH(\vp_{0,\infty})\leq h(\vp^*_{0,\infty}),
\end{align}
are satisfied for the topological entropies of $\vp_{0,\infty}$, in the sense of Definitions~\ref{d:hCMspa}, \ref{d:hCMsep} and \ref{d:hH}, and $h(f_{0,\infty})$ of $f_{0,\infty}$, $h(\vp^*_{0,\infty})$ of $\vp^*_{0,\infty}$, with regard to \eqref{eq:r:3.7} in Remark~\ref{r:3.7}.
\end{prop}
\begin{proof}
According to Theorem~\ref{t:3.1}, we have (cf. Remark~\ref{r:3.7}),
\[
\hCM^{\spa}(\vp_{0,\infty}) \leq \hCM^{\sepp}(\vp_{0,\infty})\leq h(f_{0,\infty})
\]
jointly with
\[
h(f_{0,\infty}) = \hH(\vp_{0,\infty}),
\]
because
\[
s_{\H}(\vp_{0,\infty},p,\ep,n) - \card A \leq s(f_{0,\infty},p,\ep,n) \leq s_{\H}(\vp_{0,\infty},p,\ep,n) + \card A
\]
holds as a consequence of the hypotheses \eqref{eq:p:6.1a}.

After all, we get
\[
\hCM^{\spa}(\vp_{0,\infty}) \leq \hCM^{\sepp}(\vp_{0,\infty})\leq h(f_{0,\infty}) = \hH(\vp_{0,\infty}).
\]

In view of Theorem~\ref{t:5.2},
\[
\hH(\vp_{0,\infty}) \leq h(\vp^*_{0,\infty})
\]
still holds, which already leads to \eqref{eq:p:6.1b}.
\end{proof}

\begin{example}\label{e:6.1}
As an illustrative example of the application of Proposition~\ref{p:6.1}, we can consider the u.s.c. map $\vp=\vp_j:[0,1]\to\KK([0,1])$, $j\in\N\cup\{0\}$, where $\vp=f\cup f_{0,1}$,
\[
f(x):=
\begin{cases}
2x,& \mbox{ for $x\in[0,\frac12)$,}\\
2(1-x),& \mbox{ for $x\in[\frac12,1]$,}
\end{cases}
\]
is the standard tent map and
\[
f_{0,1}(x):=
\begin{cases}
0,& \mbox{ for $x\in(0,1)$,}\\
\{0,1\},& \mbox{ for $x\in\{0,1\}$.}
\end{cases}
\]

Since 
\begin{align*}
\max_{i\in\{0,\ldots,n-1\}} \{\dH(\vp_i\circ\ldots\circ\vp_0(x),\vp_i\circ\ldots\circ\vp_0(y))\}\\
= \max_{i\in\{0,\ldots,n-1\}} \{\dH(f_i\circ\ldots\circ f_0(x),f_i\circ\ldots\circ f_0(y))\}
\end{align*}
holds for all $x,y\in[0,1]\setminus \{0,1\} = (0,1)$, i.e. $\card A=2$, the condition \eqref{eq:p:6.1a} is trivially satisfied. Hence, applying Proposition~\ref{p:6.1}, we arrive at the inequalities \eqref{eq:p:6.1b}.

Moreover, since also
\begin{align*}
&\max_{i\in\{1,\ldots,n\}} \{ \dH (\pi_i(\orbn(\vp_{0,\infty};x),\pi_i(\orbn(\vp_{0,\infty};y))))\}\\
&= \max_{i\in\{0,\ldots,n-1\}} \{ \dH( f^{i}(x)\cup\{0,1\},f^{i}(y)\cup\{0,1\})\} = \max_{i\in\{0,\ldots,n-1\}} d(f^{i}(x),f^{i}(y))
\end{align*}
holds for all $x,y\in[0,1]\setminus\{0,1\} = (0,1)$, which again leads to the equalities
\[
s_{\mathfrak{i}}(\vp,d,\ep,n)-2 \leq s(f,d,\ep,n) \leq s_{\mathfrak{i}}(\vp,d,\ep,n)+2,
\]
where $d$ is the Euclidean metric on $[0,1]$, we still get that $\hH(\vp)=\hi(\vp)=h(f)$.

Summing up, the relations
\begin{align*}
\hCM^{\spa}(\vp) \leq \hCM^{\sepp}(\vp)\leq \log 2 = h(f) = \hH(\vp) = \hi(\vp) \leq h(\vp^*)
\end{align*}
are satisfied, which partially answers also an open problem posed in Remark~\ref{r:5.1}.

On the other hand,
\[
\hrsp(\vp) = \hrse(\vp)= \hCM^{\sepp}(\vp) = \hCM^{\spa}(\vp) = h(0) = 0,
\]
according to Theorem~\ref{t:3.1}, which already trivially implies that $\hCM^{\sepp}(\vp) \leq h(\vp^*)$.

\end{example}

\begin{example}\label{e:6.2}
Consider the u.s.c. maps $\vp_{\frac12}:[0,1]\to\KK([0,1])$ and $\vp_1:[0,1]\to\KK([0,1])$, where
\begin{align*}
\vp_{\frac12}(x):=
\begin{cases}
2x,& \mbox{ for $0\leq x< \frac12$,}\\
\{0,1\},& \mbox{ for $x=\frac12$,}\\
2x-1,& \mbox{ for $\frac12<x\leq 1$,}
\end{cases}
\mbox{ and }
\vp_{1}(x):=
\begin{cases}
2x,& \mbox{ for $0\leq x< \frac12$,}\\
[0,1],& \mbox{ for $x=\frac12$,}\\
2x-1,& \mbox{ for $\frac12<x\leq 1$.}
\end{cases}
\end{align*}

Observe that $\vp_{0}:[0,1]\to[0,1]$, where
\begin{align*}
\vp_0(x):=
\begin{cases}
2x,& \mbox{ for $0\leq x \leq \frac12$,}\\
2x-1,& \mbox{ for $\frac12<x\leq 1$,}
\end{cases}
\end{align*}
is a discontinuous (at $x=\frac12$) single-valued selection of both $\vp_{\frac12}$ and $\vp_1$, $\vp_0\subset \vp_{\frac12} \subset \vp_1$, i.e. $\vp_0(x)\in\vp_{\frac12}(x)\subset\vp_1$, for every $x\in[0,1]$.

In \cite[Examples 4.8 and 4.11]{CMM} (cf. also \cite[Theorem 3.2]{CMM}), it was shown that
\[
h(\vp_0) = \hCM^{\sepp}(\vp_{\frac12}) = \hCM^{\spa}(\vp_{\frac12}) = \hCM^{\sepp}(\vp_{1}) = \hCM^{\spa}(\vp_{1}) = \log 2.
\]

One can readily check that (cf. Lemma~\ref{l:hCMsel})
\begin{align*}
0 &= \hCM^{\sepp}(\vp_{0}\cup\{0\}) = \hCM^{\spa}(\vp_{0}\cup\{0\}) = \hCM^{\sepp}(\vp_{\frac12}\cup\{0\}) = \hCM^{\spa}(\vp_{\frac12}\cup\{0\}) \\&
= \hCM^{\sepp}(\vp_{1}\cup\{0\}) = \hCM^{\spa}(\vp_{1}\cup\{0\}) \\
&\leq \min \{ h((\vp_0\cup\{0\})^{*}), h((\vp_{\frac12}\cup\{0\})^{*}), h((\vp_1\cup\{0\})^{*})\}.
\end{align*}

Now, we will consider the nonautonomous case for $\psi_{0,\infty}$, where $\bar{j}>0$ and
\begin{align*}
\psi_j :=
\begin{cases}
\vp_1, &\mbox{ for $j=0,1,\ldots, \bar{j}$,}\\
\vp_{\frac12}\cup\{0\}, &\mbox{ for $j>\bar{j}$.}
\end{cases}
\end{align*}

Observe that 
\[
\dH(\psi^{[j]}(x),\psi^{[j]}(y)) = d(\vp_0^{j}(x), \vp_0^{j}(y))
\]
holds for every $j\in\N\cup\{0\}$ and all $x,y\in[0,1]\setminus A_{\bar{j}}$, where the exceptional set $A_{\bar{j}}$ takes the form
\[
A_{\bar{j}} = \{ \frac1{2^{\bar{j}}}, \frac2{2^{\bar{j}}},\ldots, 1 - \frac1{2^{\bar{j}}} \},
\]
i.e $\card A_{\bar{j}} = 2^{\bar{j}}-1$ is finite.

Hence, applying Proposition~\ref{p:6.1}, we obtain \eqref{eq:p:6.1b} with $\hH(\psi_{0,\infty}) = h(\vp_0) = \log 2$, i.e. (cf. Lemma~\ref{l:hCMsel}) 
\[
0 = \hCM^{\spa}(\psi_{0,\infty}) = \hCM^{\sepp}(\psi_{0,\infty}) < \log 2 \leq h(\psi_{0,\infty}^*).
\]

The same results can be obtained, when replacing $\vp_1$ by $\vp_{1}\cup\{0\}$ or by $\vp_{\frac12}$ in the definition of $\psi_j$, for $j=0,1,\ldots,\bar{j}$.

If, in particular, $\psi_j=\vp_{\frac12}\cup\{0\}=\vp_0\cup\{0\}$, for every $j\in\N\cup\{0\}$, then
\[
\dH((\vp_{\frac12}\cup\{0\})^j(x),(\vp_{\frac12}\cup\{0\})^j(y)) = d(\vp_0^{j}(x),\vp_0^{j}(y))
\]
holds for every $j\in\N\cup\{0\}$ and all $x,y\in[0,1]$, i.e. even with $A\neq\emptyset$, and subsequently also
\[
s_{\H}(\vp_{\frac12}\cup\{0\},d,\ep,n) = s(\vp_0,d,\ep,n).
\]
After all, again
\[
0 = \hCM^{\sepp}(\vp_{\frac12}\cup\{0\}) < \log 2 = h(\vp_0) = \hH(\vp_{\frac12}\cup\{0\}) \leq h((\vp_{\frac12}\cup\{0\})^*),
\]
i.e. the strict inequality
\[
\hCM^{\sepp}(\vp_{\frac12}\cup\{0\}) < h((\vp_{\frac12}\cup\{0\})^*).
\]

\end{example}

Finally, we will show that the Hausdorff metric $\dH$, and subsequently the topological entropy $\hH$, can be often replaced by the Borsuk metric $\dB$, defined in \cite{Br}, and the related topological entropy denoted as $\hB$, provided the values $\vp_j^{k}(x)$ of the maps $\vp_j^k: X\to \KK(X)$, $j\in\N\cup\{0\}$, $k\in\N$, in a compact metric space $X$, are compact absolute neighbourhood retracts.

Let us recall that a metric space $\vp$ is an \emph{absolute neighbourhood retract} (shortly, $\ANR$) if, for every space $Z$ and every closed subset $A\subset Z$, each continuous map $f:A\to Y$ is extendable over some open neighbourhood of $A$ in $Z$. For more details about $\ANR$-spaces, see e.g. \cite[Chapter I.2]{AG}, \cite[Chapter 5.7]{HP}.

Hence, let $(Y,d)$ be a compact metric space and $(\KK(Y),\dH)$ be the related hyperspace defined in Section~2. By the \emph{Borsuk metric} $\dB$, we mean
\begin{align*}
\dB(C,D):= \inf\{&\ep>0: \mbox{$\exists$ continuous maps $f:C\to D$ and $g:D\to C$}\\ 
&\mbox{such that $d(x,f(x))\leq\ep$, for every $x\in C$,}\\
&\mbox{and $d(y,g(y))\leq\ep$, for every $y\in D$}\},
\end{align*}
for any $C,D\in\KK(Y)$.

If $C,D$ are not necessarily compact and $\dH(C,D)=0$, resp. $\dB(C,D)=0$, need not imply that $C=D$, then we speak about the respective \emph{pseudometrics}.

It is well known that the Borsuk continuity implies the Hausdorff continuity, but not vice versa in general. This claim follows from the inequality (see e.g. \cite[p. 25]{Gr})
\[
\dB(C,D)\geq\dH(C,D), \mbox{ for any $C,D\in\KK(Y)$,}
\]
and in particular
\[
\dB(\vp(x),\vp(y)) \geq \dH(\vp(x),\vp(y)),
\]
for all $x,y\in X$.

By the same inequality, we obtain that $\hB(\vp)\geq\hH(\vp)$, resp. $\hB(\vp_{0,\infty})\geq\hH(\vp_{0,\infty})$, hold for the topological entropies $\hH(\vp)$ of $\vp$, resp. $\hH(\vp_{0,\infty})$ of $\vp_{0,\infty}$, in the sense of Definition~\ref{d:hH}, and $\hB(\vp)$ of $\vp$, resp. $\hB(\vp_{0,\infty})$ of $\vp_{0,\infty}$, which can be simply defined when just replacing the metrics $\dH$ by $\dB$ in Definition~\ref{d:hH}.

On the other hand, if the values $\vp^k_j(x)$ of the maps $\vp^k_j:X\to\KK(X)$, $j\in\N\cup\{0\}$, $k\in\N$, are compact absolute neighborhood retracts, then we get the equality (see \cite[Theorem 5.13]{Br})
\[
\dB(C,D) = \dH(C,D), \mbox{ for any $C,D\in\KK(X)$,}
\]
and subsequently $\hB(\vp)=\hH(\vp)$, resp. $\hB(\vp_{0,\infty})=\hH(\vp_{0,\infty})$.

In particular, it occurs provided $X$ consists of, for instance, finitely many points.

Summing up, we can formulate the following proposition.

\begin{prop}\label{p:6.2}
Let $(X,d)$ be a compact metric space and $\vp_{0,\infty}$ be a sequence of u.s.c. maps $\vp_j:X\to\KK(X)$, $j\in\N\cup\{0\}$, such that the values $\vp^k_j(x)$, $j\in\N\cup\{0\}$, $k\in\N$, are compact $\ANR$-spaces, for all $x\in X$. Then the equality
\[
\hB(\vp_{0,\infty}) = \hH(\vp_{0,\infty})
\]
holds for the topological entropies $\hB(\vp_{0,\infty})$, $\hH(\vp_{0,\infty})$ of $\vp_{0,\infty}$.
\end{prop}

Proposition~\ref{p:6.2} can be directly applied to the foregoing Examples~\ref{e:6.1} and \ref{e:6.2}.

\begin{example}[{continued Example~\ref{e:6.1}}]\label{e:6.3}
Since $\vp=\vp_j:[0,1]\to\KK([0,1])$, $j\in\N\cup\{0\}$, are under the assumptions in Example~\ref{e:6.1} u.s.c., and $\vp^k(x)\in\{0,1,f^k(x)\}\subset\ANR$, $k\in\N$, for all $x\in X$, we obtain, by means of Proposition~\ref{p:6.2}, the relations
\[
h(\vp^*)\geq\hH(\vp) = \hB(\vp) = \log 2 > \hCM^{\sepp}(\vp) = 0.
\]
\end{example}

\begin{example}[{continued Example~\ref{e:6.2}}]\label{e:6.4}
Since $\psi_j:[0,1]\to\KK([0,1])$, resp. $\vp_{\frac12}\cup\{0\}=\vp_0\cup\{0\}:[0,1]\to\KK([0,1])$, $j\in\N\cup\{0\}$, are under the assumptions in Example~\ref{e:6.2} u.s.c., and $\psi^k_j(x)\in\{\{0\},[0,1],\vp_0^k(x)\}\subset\ANR$ as well as $(\vp_{\frac12}\cup\{0\})^k(x)\in\{\{0\},\vp^k_0(x)\}\subset\ANR$, $k\in\N$, for all $x\in X$, we obtain, by means of Proposition~\ref{p:6.2}, the relations
\begin{align*}
\hCM^{\sepp}(\psi_{0,\infty}) < \log 2 = \hB(\psi_{0,\infty}) = \hH(\psi_{0,\infty}) \leq h(\psi^*_{0,\infty}),
\end{align*}
resp. in particular for $\psi_j=\vp_{\frac12}\cup\{0\}=\vp_0\cup\{0\}$, $j\in\N\cup\{0\}$,
\[
\hCM^{\sepp}(\vp_{\frac12}\cup\{0\})< \log 2 = \hB(\vp_{\frac12}\cup\{0\}) = \hH(\vp_{\frac12}\cup\{0\})\leq h((\vp_{\frac12}\cup\{0\})^*).
\]
\end{example}

\begin{remark}
Observe that if, in Examples~\ref{e:6.2} and~\ref{e:6.4}, $\psi_j=\vp_1\cup\{0\}$ or $\psi_j=\vp_1$, for every $j\in\N\cup\{0\}$, then Proposition~\ref{p:6.1} cannot be applied. Let us note that in \cite{WK} the lower estimate $\hKT(\vp_1)\geq \log 3$ was obtained and so, in view of Lemma~\ref{l:3.2},
\[
\hKT(\vp_1\cup\{0\})\geq \hKT(\vp_1)\geq \log 3.
\]
Furthermore, just by the same Lemma~\ref{l:3.2}, (cf. Examples~\ref{e:6.2} and \ref{e:6.4})
\begin{align*}
\hKT(\psi_{0,\infty})&\geq h(\vp_0)=\log 2,\\
\hKT(\vp_{\frac12}\cup\{0\})&=\hKT(\vp_0\cup\{0\}) \geq h(\vp_0) = \log 2.
\end{align*}
On the other hand, as already pointed out in Remark~\ref{r:5.1}, the inequality $\hKT(\vp)\leq h(\vp^*)$ need not be satisfied for a general $\vp$. 
\end{remark}
\section{Concluding remarks}

Proposition~\ref{p:6.1} can be still improved by means of \cite[Theorem~2.7]{Sh} (in an autonomous case, cf. \cite{BS,KO}) in the following way.

If, in particular, $(X,d)$ is a compact metric space and the single-valued selections $f_j\subset\vp_j$, $j\in\N\cup\{0\}$, are equi-continuous, then by means of \cite[Theorem 2.7]{Sh}, $h(f^*_{0,\infty}) = \infty$, provided $h(f_{0,\infty})>0$.

Thus, since the equalities $h(f_{0,\infty})=\hH(\vp_{0,\infty})$ as well as $h(f^*_{0,\infty})=h(\vp^*_{0,\infty})$ are satisfied under the assumptions of Proposition~\ref{p:6.1}, we obtain for $h(f_{0,\infty})>0$ not only $h(f^*_{0,\infty})=\infty$, but also $h(\vp^*_{0,\infty})=\infty$.

Subsequently, since in Examples~\ref{e:6.1} and~\ref{e:6.3} we have
\[
\hi(\vp)=\hH(\vp)=\hB(\vp)=h(f) = \log 2,
\]
despite of $\hCM^{\sepp}(\vp)=0$, we get $h(f^*)=h(\vp^*)=\infty$, and trivially $\hKT(\vp) \leq h(\vp^*)$.

On the other hand, it is not clear whether or not the inequalities 
\[
\hCM^{\sepp}(\vp)\leq h(\vp^*), \mbox{ resp. }\hCM^{\sepp}(\vp_{0,\infty})\leq h(\vp^*_{0,\infty}),
\]
or the inequalities
\[
\hCM^{\sepp}(\vp)\leq \hH(\vp), \mbox{ resp. }\hCM^{\sepp}(\vp_{0,\infty})\leq \hH(\vp_{0,\infty})
\]
hold for general u.s.c. maps $\vp,\vp_j:X\to\KK(X)$, $j\in\N\cup\{0\}$.

If the latter particular case is not true, then the entropies $\hCM^{\sepp}$ are obviously uncomparable with those of $\hH$ in general.

Otherwise, $\hL(\vp_{0,\infty}) \leq h(\vp^*_{0,\infty})$, but not necessarily $\hU(\vp_{0,\infty})\leq h(\vp^*_{0,\infty})$, would be also satisfied. For $\hU(\vp_{0,\infty})\leq h(\vp^*_{0,\infty})$, we would even obtain that
\[
\hL(\vp_{0,\infty}) = \hCM^{\spa}(\vp_{0,\infty})\leq \hCM^{\sepp}(\vp_{0,\infty}) \leq \hU(\vp_{0,\infty}) \leq h(\vp^*_{0,\infty}).
\]

As a conclusion, we have shown in this paper that, apart from the entropies $\hH$ and $\hrsp$, $\hrse$, for which the desired inequalities $\hrsp(\vp_{0,\infty})\leq \hrse(\vp_{0,\infty})\leq \hH(\vp_{0,\infty}) \leq h(\vp^*_{0,\infty})$ are satisfied, and so resembling such a property of the single-valued case, we can also consider for the same goal, under suitable additional restrictions, some further definitions of topological entropy like those considered in \cite{CMM}. Moreover, the multivalued maps under consideration $\vp_j:X\to\KK(X)$, $j\in\N\cup\{0\}$, can be u.s.c. and the induced single-valued hypermaps $\vp^*_j:\KK(X)\to\KK(X)$, $j\in\N\cup\{0\}$, can be even discontinuous, unlike to a single-valued case. That is why we can speak with this respect about a multivalued effect (whence the title of our paper).

\section*{Acknowledgements}
The authors were supported by the Grant IGA\_PrF\_2024\_006 ``Mathematical Models'' of the Internal Grant Agency of Palack\'y University in Olomouc.
.

\bibliography{AL-EntropyOfDiscontinuous-arxiv}\bibliographystyle{siam}

\end{document}